\documentclass[11pt]{amsart}
\usepackage{hyperref}
\usepackage{amsfonts}
\usepackage{amssymb}
\usepackage{graphicx}
\usepackage{pstricks}
\usepackage{amsmath}
\usepackage{amsxtra}
\usepackage{bm, fullpage, tikz}
\usepackage{enumitem}

\usepackage{mathrsfs, amsthm, xifthen, verbatim, mathrsfs}
\usepackage[normalem]{ulem}

\usepackage{moreverb}

\usepackage[arrow, matrix]{xy}

\usepackage{multicol}

\newtheorem{thm}{Theorem}[section]
\newtheorem{lemma}[thm]{Lemma}
\newtheorem{prop}[thm]{Proposition}
\newtheorem{crl}[thm]{Corollary}

\theoremstyle{definition}
\newtheorem{dfn}[thm]{Definition}
\newtheorem{exm}[thm]{Example}

\newtheorem{rem}[thm]{Remark}

\newcommand{\reals}{\mathbb{R}}
\newcommand{\B}{\mathcal{B}}
\newcommand{\naturals}{\mathbb{N}}

\newcommand{\rx}{\sgr{\mathbb{R}}{\ux}}

\newcommand\ddfrac[2]{\frac{\displaystyle #1}{\displaystyle #2}}

\newcommand{\ux}{\underline{X}}

\newcommand{\spn}{\text{Span}}

\numberwithin{equation}{section}

\def\ker{\operatorname{ker}}

\def\rank{\operatorname{rank}}
\def\dim{\operatorname{dim}}

\def\max{\operatorname{max}}

\newcommand{\K}[1]{\mathcal{K}_{#1}}

\newcommand{\N}[1]{\mathcal{N}(#1)}
\newcommand{\X}[1]{\mathcal{X}(#1)}

\newcommand{\St}[1]{\mathcal{S}(#1)}
\newcommand{\sgr}[2]{#1[#2]}
\newcommand{\ringsop}[2]{\sum #1^{#2}}

\newcommand{\norm}[2]{\|\ifthenelse{\isempty{#2}}{\cdot}{#2}\|_{#1}}
\newcommand{\cl}[2]{\overline{#2}^{\ifthenelse{\isempty{#1}}{}{#1}}}
\newcommand{\Cnt}[2]{\mathrm{C}_{\ifthenelse{\isempty{#1}}{}{#1}}(#2)}
\newcommand{\Psd}[2]{\mbox{Pos}_{\ifthenelse{\isempty{#1}}{}{#1}}(#2)}
\newcommand{\Sp}[2]{\mathfrak{sp}_{\ifthenelse{\isempty{#1}}{}{#1}}(#2)}
\newcommand{\M}[2]{\mathcal{M}_{#1}(#2)}

\newcommand{\map}[3]{#1:#2\longrightarrow #3}

\newcommand\copyrighttext{%
  \footnotesize 
\textcopyright 2023. The final publication of this manuscript is available at J. Operator Theory 90:1(2023), 223--261, \href{https://doi.org/10.7900/jot.2021nov26.2392}{https://doi.org/10.7900/jot.2021nov26.2392}}
\newcommand\copyrightnotice{%
\begin{tikzpicture}
\node[text width=15cm, anchor=south] at (current page.south) {\fbox{\parbox{\dimexpr\textwidth-\fboxsep-\fboxrule\relax}{\copyrighttext}}};
\end{tikzpicture}%
}

\begin{document}

\title[The Truncated Moment Problem for Commutative $\reals$-Algebras]{The Truncated Moment Problem \\ for Unital Commutative $\reals$-Algebras}

\author[R.E. Curto, M. Ghasemi, M. Infusino, {\protect \and} S. Kuhlmann]{Ra\'ul E. Curto, Mehdi Ghasemi, Maria Infusino, {\protect
\and} Salma Kuhlmann}
\address{R.E. CURTO, Department of Mathematics, University of Iowa, Iowa City, 52246, USA}
\email{raul-curto@uiowa.edu}
\address{M. GHASEMI, Department of Mathematics and Statistics, University of Saskatchewan, Saskatoon, SK, S7N 5E6, Canada}
\email{mehdi.ghasemi@usask.ca}
\address{M. INFUSINO, Dipartimento di Matematica e Informatica, Universit\'{a} degli Studi di Cagliari, Palazzo delle Scienze, Via Ospedale 72, 09124 Cagliari}
\email{maria.infusino@unica.it}
\address{S. KUHLMANN, Fachbereich Mathematik und Statistik, Universit\"{a}t Konstanz, Universit\"{a}tstrasse 10, 78457 Konstanz, Germany}
\email{salma.kuhlmann@uni-konstanz.de }
\begin{abstract} 
We investigate when a linear functional $L$ defined on a linear subspace $B$ of a unital commutative real algebra $A$ admits an integral representation w.r.t.\! a positive Radon measure supported on a closed subset $K$ of the character space of $A$. We provide a criterion for the existence of such a representation for $L$ when $A$ is equipped with a submultiplicative seminorm. We then build on this result to prove our main theorem for $A$ not necessarily equipped with a topology. This allows us to extend well-known {classical} results on truncated moment problems. 
\end{abstract}


\keywords{truncated moment problem, full moment problem, integral representation, linear functional. \\ \vspace{0.3cm} \copyrightnotice}

\maketitle

\emph{This paper is dedicated to the memory of Murray Marshall, who passed away in May 2015 and whose influence can be felt vividly throughout the article.}

\section*{\sc{\large{Introduction}}}

The Classical Truncated Moment Problem (TMP) dates back to the last decade of the nineteenth century, and was initially developed by a number of mathematicians, including T.J. Stieltjes, P. Chebishev, H. Hamburger, A.A. Markov, N.I. Akhiezer, M.G. Krein, A.A. Nudel'man, J.A. Shohat, J.D. Tamarkin, M. Riesz and I.S. Iohvidov. The theory ran parallel to the developments in the full moment problem, where the main focus was placed. Many decades later, renewed interest in TMP arose in connection with the so-called Subnormal Completion Problem (SCP) for unilateral weighted shifts. In 1966, J. Stampfli \cite{Sta} proved that for any three positive numbers $a<b<c$, it is always possible to build a unilateral weighted shift $W_{\alpha}$ acting on $\ell^2(\naturals_0)$, with $\alpha \in \ell^{\infty}(\naturals_0$), having initial weights $\alpha_0=a$, $\alpha_1=b$, $\alpha_2=c$, and such that $W_{\alpha}$ is subnormal. In \cite{Houston,RGWSI}, R.E. Curto and L.A. Fialkow solved the SCP for unilateral weighted shifts, by finding necessary and sufficient conditions for a finite collection of positive numbers to be the initial segment of weights of a subnormal unilateral weighted shift. Their approach was based on the fact that subnormality is detected by the existence of a positive Radon measure on the closed interval $[0,\left\|W_{\alpha}\right\|^2]$ whose moments are the moments $\gamma_k$ of the weight sequence $\alpha$, defined recursively as $\gamma_0:=1$ and $\gamma_{k+1}:=\alpha_k^2 \gamma_k \; (\textrm{for all } k \in \naturals_0)$. Thus, the subnormality of $W_{\alpha}$ is intrinsically related to a TMP. In the process, Curto and Fialkow proved the so-called Flat Extension Theorem for moment matrices, which is an essential component of their TMP theory in one and several real or complex variables. 

A few years after the Curto-Fialkow results were published, J.B. Las\-serre discovered some significant connections between real algebraic geometry, moment problems and polynomial optimization; he introduced a method known as semidefinite relaxations (see, e.g., \cite{JBL01}), which led to renewed interest in solutions of TMP, especially those with finitely atomic representing measures. The importance of polynomial optimization problems and the convenience of working with polynomials as algebraic and computational objects as well as intensive research on this area, is one of the main motivations for the study of moment problems for the algebra of polynomials. 
For ample information on the above mentioned developments, the reader is referred to \cite{Akh,tcmp1,tcmp3,tcmp4,tcmp8,Curto-Fialkow-2008,HeNi,HeLa,IKLS2017,Ioh,KimWoe,KrNu,JBL01,LassBook,Lau2,Lau3,LaMo,Put,Sch,Zalar}. 

\label{page2} A fundamental tool in all those works is positivity. Given a closed subset $K$ of $\reals^n$, a linear functional $L$ defined on a subspace of $\rx:=\reals[X_1,\ldots,X_n]$ is said to be \emph{$K-$positive} when it assumes nonnegative values in all the elements
of its domain which are nonnegative on $K$. For a set $\mathcal{A}$ of monomials in $\rx$, a closed subset $K$ of $\reals^n$, and a $K-$positive linear functional $L$ on the $\spn(\mathcal{A})$, the $\mathcal{A}$--truncated $K$--moment problem is the question of establishing whether $L$ can be represented as an integral with respect to a positive Radon measure whose support is contained in $K$. If such a measure exists then it is called a $K$--representing measure for $L$. The $\mathcal{A}-$Truncated $K-$Moment Problem terminology was introduced by J. Nie in \cite{Nie}, although he only considered the case when the set $\mathcal{A}$ is finite. When $\mathcal{A}=\{\ux^{\alpha}: \alpha\in\naturals_0^n, |\alpha|\leq d\}$, for some $d \in \naturals,$ the $\mathcal{A}$--truncated $K$--moment problem is usually known as the Classical $K$--TMP. 

The Full $K$--Moment Problem, for closed $K\subseteq\reals^n$, corresponds to the case when $\mathcal{A}=\{\ux^{\alpha}: \alpha\in\naturals_0^n \}$; that is, given a $K-$positive linear functional $L$ on $\rx$, find a criterion for the existence of a positive Radon measure $\mu$ whose support is contained in $K$, and such that $L$ is represented as $L(p)=\int p~d\mu$, for all $p\in\rx$.   

Partial answers to the $\mathcal{A}$--truncated $K$--moment problem are known. For example, when $K$ is a closed and non-compact subset of $\reals^n$ and $\mathcal{A}$
is the set of all monomials up to a certain degree $2d$ or $2d+1$, the existence of such a $K$--representing measure is proved
to be equivalent to the $K$--positive extendability of $L$ to the set of all polynomials of degree at most $2d+2$, \cite[Theorem 2.2]{Curto-Fialkow-2008}. When $K$ is compact and $\spn(\mathcal{A})$ contains a polynomial that is strictly positive on $K$, the existence of a $K$--representing measure is known to be equivalent to the $K$--positivity of $L$ (see \cite[Theorem I, p.129]{Tchakaloff}, \cite[Theorem 2.2]{Fialkow-Nie} and \cite[Algorithm 4.2]{Nie}). 

In the present article, we show that all the above solutions can be considered as particular cases of {two general results for linear functionals defined on a linear subspace of arbitrary dimension of a unital commutative $\reals$--algebras in lieu of finite 
dimensional subspaces of $\rx$}. Thus, our setting is general enough to encompass also infinite dimensional instances of the moment problem, e.g. when the algebra is not finitely generated or when the representing measure is supported on an infinite dimensional vector space. 

Infinite dimensional moment problems have been studied already in the sixties (see e.g. \cite{Col63, P64, GP64, K67, BS71, BY75, L75, BK88, Schm90}) motivated by fundamental questions in applied areas such as statistical physics and quantum mechanics. Since then there has been an extensive production on the infinite dimensional moment problems appearing in the analysis of interacting particle systems as well as in stochastic geometry, spatial ecology, neural spike trains, heterogeneous materials and random packing (see, e.g., \cite{AH08, KKO06, LM15, M04, BKM04, ST04, TS06}). Despite the vast literature devoted to the theory of the infinite dimensional moment problem, and more generally of the moment problem on unital commutative algebras (see \cite{KuLeSp11, GKM14, AJK15, CIK16, GKM16, KKL15, IKM18, Schm18, IK2020, IKKM-ArX}, just to mention a few recent developments), several questions remain open (cf. \cite{IK2017}). Let us briefly describe our contribution towards a better knowledge of the TMP in this general setting.

In Section \ref{CntPosExt}, notations and basic definitions are introduced. In Section \ref{GTMP}, we give the definition of the $B$--truncated $K$--moment problem, which is a generalization of the Classical truncated $K$--moment problem for linear functionals $L$ defined on a linear subspace $B$ of a unital commutative $\reals$--algebra $A$ and a closed subset $K$ of characters of $A$ (see Definition \ref{B-TKMP}). 
In Theorem \ref{posext} we prove a first criterion for solving this problem when a seminormed structure is present on $A$, employing techniques from the theory of positive extensions of linear functionals. The rest of the section deals with algebras that are not necessarily equipped with a seminorm, but that are considered in a setting which implies the existence of a seminormed structure. Theorem \ref{cmpttrnctmmnt} (respectively Theorem~\ref{trnctmmnt}) solves the $B$--truncated $K$--moment problem (see Definition \ref{B-TKMP}) for $K$ compact (respectively $K$ closed and non-compact) subset of characters of $A$. 

In Section \ref{SupportRepMeas} we use the Choquet-Bishop-de Leeuw Theorem to determine the nature of the support of the representing measures obtained in previous sections, in the case when $B$ is a finite dimensional subspace of the algebra of continuous real-valued functions on a compact space (Theorems \ref{generalizedTchakaloff} and \ref{thm42}).

In Section \ref{FullVSTruncated} we prove an analogue of J. Stochel's theorem \cite{JSt} about the
relation between the full moment problem and the truncated moment problem 
(see Theorem \ref{Stochel}). The notion of truncated $K$--frames is introduced to treat Stochel's result in our general setting. We then use Theorems~{\ref{cmpttrnctmmnt} and} \ref{trnctmmnt} and Theorem \ref{Stochel} to obtain a new proof of the main result in M. Putinar and F.-H. Vasilescu's paper on dimensional extensions \cite[Theorem 2.5]{PuVa} (see Theorem \ref{PuVaRevisited}), and also an alternative proof of a result of M. Marshall on localization of $\rx$ at the multiplicative set of powers of a positive polynomial \cite[Corollary~6.2.4]{MurrayBook} (see Theorem \ref{thm56}). 

Finally, we reserve Section \ref{Applications} for a discussion of several applications of our main results, ranging from the Classical TMP to the TMP for point processes, and to the SCP for $2$--variable weighted shifts. In particular, in the case of bivariate polynomials, our novel approach allows us to deal with not only the cases of the Rectangular, Triangular, and Sparse Connected TMP, but also with a new hybrid case which includes the presence of infinitely many moments in one of the variables.


\section{\sc{\large{Notation and Preliminaries}}}\label{CntPosExt}
In this article all the algebras are assumed to be commutative $\reals$--algebras with unit. 

A \emph{submultiplicative seminorm} on an algebra $A$ is a map $\rho: {A}\to{[0,\infty)}$ satisfying:
\begin{enumerate}
	\item{\label{sn1}$\textrm{ for all } a\in A \; \textrm{and } r\in\reals, \quad\rho(ra)=|r|\rho(a)$ (positive homogeneous)},
	\item{\label{sn2}$\textrm{ for all } a,b\in A,\quad\rho(a+b)\leq\rho(a)+\rho(b)$ (subadditive),}
	\item{\label{sn3}$\textrm{ for all } a,b\in A,\quad\rho(ab)\leq\rho(a)\rho(b)$ (submultiplicative).}
\end{enumerate}
Note that conditions \ref{sn1}, \ref{sn2} and \ref{sn3} above imply continuity of scalar 
multiplication, addition and multiplication  on $A$, respectively. The pair $(A,\rho)$ is called 
a \emph{seminormed algebra}.

We denote by $\X{A}$ the set of all real-valued $\reals$--algebra homomorphisms on $A$
(the character space of $A$). One can identify $\X{A}$ as a subset of $\reals^A$. It is important to add the following assumption. Throughout this article we assume that $\X{A}$ is nonempty. Instances when $\X{A}$ is empty are given by the unital commutative $\mathbb{R}$--algebra $A:=\mathbb{R}[[\underline{X}]]$ of power series in $\underline{X}$ and the Arens algebra $L^{\omega}([0,1]):=\bigcup_{p>1}L^p([0,1])$ (\cite{Ar}, \cite[Ex. 2.5.5]{Schm90}).  

For any subset $S\subseteq A$ we define 
$\K{S}$ to be the common nonnegativity set of $S$ on $\X{A}$, i.e.,
\[
	\K{S}:=\{\alpha\in\X{A}:\alpha(S)\subseteq[0,\infty)\}.
\]
Every $a\in A$ induces a map $\map{\hat{a}}{\X{A}}{\reals}$ which is defined by 
$\hat{a}(\alpha)=\alpha(a)$ for every $\alpha$ in $\X{A}$.
These are basically projection maps on the $a^{\rm th}$ component of $\X{A}$ as a subset of 
$\reals^A$. We endow $\X{A}$ with the weakest topology which makes all maps $\hat{a}$ continuous. Note that this topology 
coincides with the subspace topology on $\X{A}$ inherited from $\reals^A$ equipped with the product topology. When $A = \rx$, the characters of $A$ are in one-to-one correspondence with point evaluations, i.e., $\X{A} \cong \mathbb{R}^n$, and so $\hat{p}=p$ for every $p \in A$. As a result, we will routinely omit the \ $\hat{}$ \ when dealing with elements of $\rx$. Note that in this case, for any $S \subseteq \reals[\ux]$, $\K{S}=\{x \in \reals^n: q(x) \ge 0 \textrm{ for all } q \in S \}$.  

We consider the Borel $\sigma$--algebra on $\X{A}$ generated by this topology and, given a Borel set $K\subseteq\X{A}$, we say that a Radon measure $\mu$ has support contained in $K$ if $M\cap K=\emptyset$ implies $\mu(M)=0$ for all Borel measurable subsets $M$ in $\X{A}$.

To any subset $K\subseteq\X{A}$ and any linear subspace $B\subseteq A$ we associate $\Psd{B}{K}$ defined as
\[
	\Psd{B}{K}:=\{b\in B:\hat{b}(\alpha)\ge0,~\textrm{ for all }\alpha\in K\}.
\]
$\Psd{B}{K}$ is a cone in $B$, i.e., closed under sums and multiplication by nonnegative real numbers. We also let $\sum A^2$ denote the cone of sums of squares of elements of $A$, and we similarly define $\sum A^{2d}$ for $d \in \mathbb{N}$.  
For any $d \in \naturals$, a $\ringsop{A}{2d}$--module $S\subseteq A$ is a cone containing $1$ such that 
for every $a\in A$, $a^{2d}\cdot S\subseteq S$. A $\sum A^{2d}-$ preorder is a $\sum A^{2d}-$ module which is closed under multiplication. A linear functional $L:B\to\reals$ is said to be \emph{$K$--positive} if $L(\Psd{B}{K})\subseteq [0, +\infty)$. When $B=A$ and $K=\X{A}$, we say that $L$ is positive.

If the algebra $A$ is equipped with a submultiplicative seminorm $\rho$, then the subset $\Sp{\rho}{A}$ 
of all $\rho$--continuous elements of $\X{A}$ is called the {\it Gelfand spectrum} of $(A,\rho)$ and 
\[
	\Sp{\rho}{A}=\{\alpha\in\X{A}:|\alpha(a)|\leq\rho(a) \; \textrm{ for all } a\in A\}.
\]
Since $\Sp{\rho}{A}\subseteq\prod_{a\in A}[-\rho(a), \rho(a)]$ is closed, and the latter set is 
compact in the product topology, we see that $\Sp{\rho}{A}$ is compact. Also, if $K$ is a nonempty compact subset of $\X{A}$, then for every $a\in A$, we have $\hat{a}\restriction_K\in\Cnt{b}{K}$, where $\Cnt{b}{K}$ denotes the subalgebra of bounded continuous real-valued functions on $K$. Thus, the supremum norm
\begin{equation}\label{sup-norm-K}
	\rho_K(f):=\sup_{\alpha\in K}|f(\alpha)| \; \textrm{ for all } f\in\Cnt{b}{K}.
\end{equation}
 induces a submultiplicative seminorm on $A$. More precisely, given $a \in A$ we will let 
\begin{equation} \label{eq22}
\rho_K(a):=\rho_K(\hat{a}) \; \; \textrm{for all } a \in A.
\end{equation}
We will also let $C(X)$ denote the algebra of continuous real-valued functions on a topological space~$X$.

\begin{rem}\label{compactification}
A particularly interesting instance of the Gelfand spectrum appears when $A$ is a uniform algebra, 
i.e., when $A$ can be realized as a subalgebra of bounded continuous real-valued functions, $\Cnt{b}{X}$, equipped with the supremum norm $\rho_X$ for a completely regular space $X$.
If $A\subseteq\Cnt{b}{X}$ is a unital algebra separating points of $X$, then for each $x\in X$, $e_x$ 
(the evaluation at~$x$) induces a continuous $\reals$--algebra homomorphism on $A$; hence, $e_x$ belongs to 
$\Sp{\rho_X}{A}$. It is easy to see that $x\mapsto e_x$ maps $X$ injectively into $\Sp{\rho_X}{A}$ and 
every $f\in A$ extends continuously to $\Sp{\rho_X}{A}$. As a result, $\hat{A}:=\{\hat{a}:a\in A\}$ can be regarded as a subalgebra of 
$C(\Sp{\rho_X}{A})$.) \ Moreover, $X$ is dense in 
$\Sp{\rho_X}{A}$. If not, take $\alpha\in\Sp{\rho_X}{A}\setminus\cl{}{X}$. By Urysohn's lemma,
there is $g\in\Cnt{}{\Sp{\rho_X}{A}}$ with $g(\alpha)=1$ and $g\restriction_{\cl{}{X}}=0$ (note that here we are using the map $x\mapsto e_x$, so this precisely means that $g(e_x)=0\,\,\textrm{ for all } x\in\cl{}{X}$). Since $\hat{A}$ contains constant functions and separates points of $\X{A}$, and in particular separates the points of $\Sp{\rho_X}{A}$, we have that the subalgebra $\hat{A}\restriction_{\Sp{\rho_X}{A}}:=\{\hat{a}\restriction_{\Sp{\rho_X}{A}}:a\in A\}$ of $\Cnt{b}{\Sp{\rho_X}{A}}$ satisfies the requirements of the Stone-Weierstrass Theorem. Hence, for any $\epsilon>0$, 
there exists $a_{\epsilon}\in A$ such that $\rho_X(g-\hat{a}_{\epsilon})=\sup_{x\in X}|g(e_x)-\hat{a}_{\epsilon}(e_x)|<\epsilon$.  
Choosing in particular
$\epsilon>0$ such that $\frac{1-\epsilon}{\epsilon}>1$, we get
$$
|g(\alpha)-\hat{a}_{\epsilon}(\alpha)|=|1-\alpha(a_{\epsilon})|<\epsilon,
$$
or equivalently,
$$
1-\epsilon<\alpha(a_{\epsilon})<1+\epsilon.
$$
Also, for every $x\in X$, $|g(e_x)-\hat{a}_{\epsilon}(e_x)|=|0-a_{\epsilon}(x)|<\epsilon$. Thus
\[
	1-\epsilon\leq\sup_{\alpha\in\Sp{\rho_X}{A}}|\hat{a}_{\epsilon}(\alpha)|\leq
	\rho_X(a_{\epsilon})\leq\epsilon<1-\epsilon,
\]
which is a contradiction. Therefore, $\cl{}{X}=\Sp{\rho_X}{A}$.
Conversely, every Hausdorff compactification of $X$ can be realized as the Gelfand spectrum of a suitable 
subalgebra of $\Cnt{b}{X}$ (for more details see \cite[Chapter I, Theorem 8.1]{TWG}). 
\end{rem}
\begin{dfn}\label{norm-cone}
Let $(A,\rho)$ be a seminormed algebra and $C$ be a cone of $A$. For every $a\in A$, we define 
\[
	\left\|a\right\|_{C;\rho}:=\inf_{p\in C}\rho(a+p).
\]
\end{dfn}
\begin{prop}\label{submul} \ $\left\|{}\right\|_{C;\rho}$ is a sublinear function on $A$, i.e., 
\begin{enumerate}
	\item[1.]{\label{I}
$\textrm{for all } a,b \in A,	\; \left\|{a+b}\right\|_{C;\rho}\leq\left\|{a}\right\|_{C;\rho}+\left\|{b}\right\|_{C;\rho} ,$}
and 
	\item[2.]{\label{II}
	$\textrm{for all } \lambda \ge 0, a \in A, \; \left\|{\lambda a}\right\|_{C;\rho}=\lambda\left\|{a}\right\|_{C;\rho} .$
	}
\end{enumerate}
Moreover, for all $a\in A$, $\left\|{a}\right\|_{C;\rho}\leq\rho(a)$.
\end{prop}

\begin{proof}\ \\
Let $a,b\in A$. By Definition \ref{norm-cone}, there are nets $(p_{t})_{t>0}$ and 
$(q_{t})_{t>0}$ in $C$ such that $\lim\limits_{t\rightarrow0}\rho(a+p_{t})=\left\|{a}\right\|_{C;\rho}$ 
and $\lim\limits_{t\rightarrow0}\rho(b+q_{t})=\left\|{b}\right\|_{C;\rho}$, decreasingly.
Then we get
\[
\begin{array}{lcl}
	\left\|{a+b}\right\|_{C;\rho} & = & \inf_{c\in C}\rho(a+b+c)\\
		& \leq & \lim \inf \limits_{t\rightarrow0}\rho(a+b+p_{t}+q_{t})\\ 
		& \leq & \lim \inf \limits_{\epsilon\rightarrow0}\left(\rho(a+p_{t})+\rho(b+q_{t})\right)\\
		& = & \lim\limits_{t\rightarrow0}\rho(a+p_{t}) + \lim\limits_{t\rightarrow0}
		\rho(b+q_{t})\\
		& = & \left\|{a}\right\|_{C;\rho} + \left\|{b}\right\|_{C;\rho}.
\end{array}
\]
This establishes property \emph{1}. As for property \emph{2}, let $a\in A$ and $\lambda\ge0$. If $\lambda=0$ then \eqref{II} clearly holds as $\left\|{0}\right\|_{C;\rho}=0$, where $0$ denotes the additive identity in $A$. If $\lambda>0$, then
\[
	\left\|{\lambda a}\right\|_{C;\rho} = \inf_{c\in C}\rho(\lambda a+c)
	 =  \inf_{c\in C}\rho\left(\lambda \left(a+\frac{c}{\lambda}\right)\right)
		 =  |\lambda|\inf_{c\in C}\rho \left(a+\frac{c}{\lambda}\right)
		 = \lambda \left\|{a}\right\|_{C;\rho}.
\]
Moreover, $\left\|{a}\right\|_{C;\rho}\leq\inf_{c\in C}\left(\rho( a)+\rho(c)\right)=\rho(a)+\inf_{c\in C}\rho(c)=\rho(a)$.
\end{proof}

\begin{prop}\label{nrm-psd}
Let $A$ be a commutative $\mathbb{R}$--algebra, $K$ a nonempty compact subset of $\X{A}$, and $\rho_K$ be as in (\ref{sup-norm-K}) and (\ref{eq22}). For $a \in A$, let $\hat{a}_+(\alpha):=\max\{0,\hat{a}(\alpha)\}$. Then 
$$
\norm{\ringsop{A}{2};\rho_K}{a}=\rho_K(\hat{a}_+)=\norm{\Psd{A}{K};\rho_K}{a}. 
$$
\end{prop}

\begin{proof}
Let $\hat{a}_-(\alpha)=\max\{0,-\hat{a}(\alpha)\}$. Then $\hat{a}=\hat{a}_+-\hat{a}_-$ and both functions 
$\hat{a}_+$ and $\hat{a}_-$ are nonnegative and continuous on the compact set $K$. Since $\hat{A}:=\{\hat{a}:a\in A\}$ contains constant functions and separates points of $\X{A}$, and in particular separates the points of $K$, we have that the subalgebra $\hat{A}\restriction_K:=\{\hat{a}\restriction_K:a\in A\}$ of $\Cnt{b}{K}$ satisfies the requirements of the Stone-Weierstrass Theorem. Hence, for every $\epsilon>0$, there exists $b\in A$ such that 
$\rho_K(\hat{b}-\sqrt{\hat{a}_-})<\delta$, where $\delta>0$ is chosen in a way that $|r-s|<\delta$ implies 
$|r^2-s^2|<\epsilon$ for $r,s$ in a suitable compact set containing $\hat{a}_-(K)$. The existence of such $\delta$ 
is guaranteed by uniform continuity of the square function on compact subsets of $\reals$. 
Then $b^2\in\ringsop{A}{2}\subseteq\Psd{A}{K}$ and we have
\[
	\rho_K(a+b^2)=\rho_K(\hat{a}_+-\hat{a}_-+\hat{b}^2)<\rho_K(\hat{a}_+)+\epsilon.
\]
Thus, $\norm{\ringsop{A;\rho_K}{2}}{a}\leq\rho_K(\hat{a}_+)$.

To see the reverse inequality, note that $\rho_K(a)=\max\{\rho_K(\hat{a}_+),\rho_K(\hat{a}_-)\}$ and for each 
$p\in\ringsop{A}{2}$, clearly $\rho_K(\hat{a}_+)\leq\rho_K(a+p)$, therefore, 
$\rho_K(\hat{a}_+)\leq\norm{\ringsop{A}{2};\rho_K}{a}$.
Moreover, $\cl{\rho_K}{\ringsop{A}{2}}=\Psd{A}{K}$, so {$\norm{\ringsop{A}{2};\rho_K}{a}=\norm{\Psd{A}{K};\rho_K}{a}$} for 
every $a\in A$.
\end{proof}

\begin{rem} In particular, for any $\beta \in \X{A}$, $\norm{\ringsop{A}{2};\rho_{\{\beta\}}}{a}=\hat{a}_+(\beta)=\max\{0,\beta(a)\}$. 
\end{rem}


\section{\sc{\large{General truncated moment problems}}}\label{GTMP} 
In this section we study the following generalization of the Classical truncated moment problem.

\begin{dfn}\label{B-TKMP}
Let $A$ be a unital commutative $\reals$--algebra. Given a closed subset $K$ of $\X{A}$, a linear subspace $B$ of $A$, and a linear functional $\map{L}{B}{\reals}$, the \emph{$B$--Truncated $K-$Moment Problem} asks whether there exists a positive Radon measure $\nu$
whose support is contained in $K$ such that 
\[
	L(b)=\int\hat{b}~d\nu\quad\textrm{ for all } b\in B.
\]
When this representation exists, $\nu$ is called a $K$--representing measure for~$L$.
\end{dfn}
The following is one of the main results of this article and provides a necessary and sufficient condition for the existence of a solution $\nu$ to the $B$--truncated $K$--moment problem, in the presence of a submultiplicative seminorm on $A$. 

\begin{thm}\label{posext}
Let $(A,\rho)$ be a seminormed algebra, $B$ a linear subspace of $A$, $d \in \mathbb{N}$, $S$ a $\ringsop{A}{2d}$--module, and $\map{L}{B}{\reals}$ a linear functional. Then $L$ admits an integral representation with respect to a positive Radon measure whose support is contained in 
$\Sp{\rho}{A}\cap\K{S}$ if and only if there exists a $D>0$ such that
\[
	L(b)\leq D\norm{S;\rho}{b}
\]
for all $b\in B$. \end{thm}
\begin{proof}
($\Rightarrow$)
Suppose that $L$ admits an integral representation with respect to a positive Radon measure $\mu$ supported in $\Sp{\rho}{A}\cap\K{S}$. Then, in particular, $L$ admits a $\rho$--continuous extension $\bar{L}$ to $A$ which 
is also positive on $S$. Therefore, 
\[
	\textrm{ for all } p\in S,\quad\bar{L}(p)\geq0.
\]
Clearly for every $a\in A$ and $p\in S$, we have $\hat{a}\leq\hat{a}+\hat{p}$ on $\Sp{\rho}{A}\cap\K{S}$, 
and so
\[
\begin{array}{lcl}
	\bar{L}(a) \leq  \bar{L}(a+p)	 \leq  |\bar{L}(a+p)|  \leq  \rho'(\bar{L})\rho(a+p).
\end{array}
\]
(Here, $\rho'(\bar{L})=\sup_{\rho(a)\leq1}|\bar{L}(a)|$, which is finite by 
$\rho$--continuity). Thus,
\[
	\bar{L}(a)\leq\rho'(\bar{L}) \; \inf_{p\in S}\rho(a+p)=\rho'(\bar{L}) \; \norm{S;\rho}{a} \quad (a \in A),
\]
which proves the necessity part.

($\Leftarrow$) Now assume that there exists $D>0$ such that for every $b\in B$, $L(b)\leq D\norm{S;\rho}{b}$. Since $\norm{S;\rho}{}$ is sublinear, by the Hahn-Banach Theorem \cite[Theorem 3.2]{RFA} $L$ admits an extension $\map{\bar{L}}{A}{\reals}$ such that 
$$
-D\norm{S;\rho}{-a}\leq\bar{L}(a)\leq D\norm{S;\rho}{a},
$$
for all $a\in A$.
By Proposition \ref{submul}, $D\norm{S;\rho}{-a}\leq D\rho(-a)=D\rho(a)$, thus
\[
	-D\rho(a)\leq -D\norm{S;\rho}{-a}\leq\bar{L}(a)\leq D\norm{S;\rho}{a}\leq D\rho(a),\quad\textrm{ for all } a\in A.
\]
The above inequality implies that $\bar{L}$ is $\rho$--continuous. We show that $\bar{L}$ is also nonnegative on $S$.
Let $p\in S$; then $-\bar{L}(p)=\bar{L}(-p)\leq D\norm{S;\rho}{-p}$. By definition,  
$\norm{S;\rho}{-p}=\inf_{q\in S}\rho(q-p)$, and therefore $\norm{S;\rho}{-p}=0$. So $\bar{L}(p)\geq0$ for each $p\in S$.
Applying \cite[Corollary 3.8]{GKM}, we find a positive Radon measure $\mu$ supported in $\Sp{\rho}{A}\cap\K{S}$ such that
\[
	\bar{L}(a)=\int\hat{a}~d\mu,\quad\textrm{ for all } a\in A,
\]
and hence on $B$ itself, as desired.
\end{proof}
The previous result can be extended to the case when $A$ is endowed with a \it locally multiplicatively convex \rm (lmc) topology $\tau$, i.e., the topology on $A$ generated by some family $\mathcal{F}_\tau$ of submultiplicative seminorms on $A$.  
 \ We denote by $\Sp{\tau}{A}$ the Gelfand spectrum of $(A,\tau)$, i.e., the  set of all
$\tau$--continuous $\alpha \in \X A$. Since any  linear functional on $A$ is $\tau$--continuous if and only if it is $\rho$--continuous for some $\rho \in \mathcal{F}_\tau$ (see \cite[Lemma 4.1]{GIKM}), we have that
\begin{equation} \label{eqrho}
\Sp{\tau}{A} =
\bigcup_{\rho \in \mathcal{F}_\tau} \Sp{\rho}{A}.
\end{equation} 
Without loss of generality, we can always assume that the family $\mathcal{F}_\tau$ is directed. This also implies that the union in (\ref{eqrho}) is directed by inclusion. Then it is easy to see that the following holds.

\begin{crl}
Let $A$ be a unital commutative $\reals$--algebra endowed with the lmc topology $\tau$ generated by the directed family $\mathcal{F}_\tau$ of submultiplicative seminorms on $A$. Let $B$ be a linear subspace of $A$, $d \in \mathbb{N}$, $S$ a $\ringsop{A}{2d}$--module, and $\map{L}{B}{\reals}$ a linear functional. Then $L$ admits an integral representation with respect to a positive Radon measure whose support is contained in 
$\Sp{\tau}{A}\cap\K{S}$ if and only if there exist a $D>0$  and a $\rho\in\mathcal{F}_\tau $ such that
\[
	L(b)\leq D\norm{S;\rho}{b},\quad\textrm{ for all } b\in B,
\]
where $\norm{S;\rho}{b}:=\inf_{p\in S}\rho(b+p)$.

\end{crl}
We utilize the criterion in Theorem \ref{posext} to study the existence of representing measures in the absence of a fixed seminorm. We begin by giving a slight generalization of a result due to Choquet \cite[Theorem~34.2]{chq}. For the reader's convenience, we include the proof.

\begin{lemma}\label{chq}
(Choquet's Lemma) \ Let $W$ be a linear subspace of an $\reals$--vector space $V$, let $C\subseteq V$ be a convex cone and let $W_C:=(W+C)\cap(W-C)$. 
Let $\map{\ell}{W}{\reals}$ be a linear functional with $L(W\cap C)\ge0$. 
Then $\ell$ admits an extension $\bar{\ell}$ to $W_C$ such that $\bar{\ell}(W_C\cap C)\ge0$.
\end{lemma}

\begin{proof}
It is clear that $W_C$ is a linear subspace of $V$ containing $W$. We show that the function 
$p(v):=-\sup\{\ell(w):w\in W \textrm{ and } v-w\in C\}$, $v\in W_C$, is a sublinear function such that $p|_W=-\ell$. 
To see this, note that there are $w,w'\in W$ and $c,c'\in C$ such that $v=w+c=w'-c'$. Thus, $w'-w=c+c'\in C\cap W$ 
and hence $\ell(w'-w)\ge0$, or equivalently $\ell(w)\leq L(w')$. Therefore, the set $\{\ell(w):w\in W\wedge v-w\in C\}$ is nonempty and 
bounded above. Hence, $p(v)$ exists. Clearly, $p(\lambda v)=\lambda p(v)$, so it remains to show that 
$p(v+v')\leq p(v)+p(v')$. If $v-w\in C$ and $v'-w'\in C$, then $(v+v')-(w+w')\in C+C=C$. 
Thus, $-p(v)-p(v')\leq-p(v+v')$ or equivalently $p(v+v')\leq p(v)+p(v')$. For every $v\in W$, $0=v-v\in C$, 
therefore $p(v)\leq-\ell(v)$. Also for every $w\in W$ with $v-w\in C$, we have $\ell(w)\leq \ell(v)$, because $\ell(W\cap C)\ge0$.
Therefore, $-\ell(v)\leq p(v)$, which proves $p|_W=-\ell$.

Applying the Hahn-Banach theorem, $-\ell$ admits an extension $-\bar{\ell}$ to $W_C$ such that
$-\bar{\ell}(v)\leq p(v)$ on $W_C$. For $c\in C\cap W_C$, $p(c)\leq0$ and hence $0\leq-p(c)\leq\bar{\ell}(c)$, as desired. 
\end{proof}

\begin{crl}
(Riesz-Krein Extension Theorem; cf. \cite[Theorem 3.6]{KuLeSp11}) \ Let $V$ be a vector space of real-valued functions on a set $X$ and let $V_0$ be a linear subspace that dominates $V$, i.e., for every $v\in V$ there exist $v_1, v_2 \in V_0$ such that $v_1 \leq v \leq v_2$. Then any positive linear functional on $V_0$ has at least one positive linear extension to all of $V$.
\end{crl}

\begin{proof}
Let $C=\Psd{V}{X}$. Then, for every $v\in V$ and $v_1, v_2\in V_0$ as above, we get
$v-v_1, v_2-v\in C$. Thus, $\pm v\in (V_0+C)\cap(V_0-C)$. Now apply Lemma \ref{chq} to obtain the desired positive extension.
\end{proof}

\begin{rem}
While the Riesz-Krein Extension Theorem can be derived from Choquet's Lemma, we do not know whether the converse is true. 
\end{rem} 

We are now ready to state and prove the two main results of this section.

\begin{thm}\label{cmpttrnctmmnt} \ Let $A$ be a unital commutative $\reals$--algebra. Suppose $K\subseteq\X{A}$ is compact and $B\subseteq A$ is a linear subspace such that there exists $q\in B$ with $\hat{q}$ strictly 
positive on $K$. Then every $K-$positive linear functional $\map{L}{B}{\reals}$ admits 
an integral representation via a positive Radon measure supported in $K$.
\end{thm}

\begin{proof}
Since $\hat{q}$ is strictly positive on $K$ and $K$ is compact, a scalar multiple of $q$, which we again denote by $q$, is such that $\hat{q}\geq 1$ on $K$. By the compactness of $K$, every $b\in B$ induces a bounded continuous 
function $\hat{b}$ on $K$, so we define $\iota: B\to C(K)$ by $\iota(b)=\hat{b}$. Let \[
\begin{array}{cccl}
	\tilde{L}: &\iota(B)&\longrightarrow &\reals\\
	 & {\hat{b}} &\mapsto& {L}(b),
	\end{array}
\]
{then the $K-$positivity of $L$ implies the $K-$positivity of $\tilde{L}$. Therefore, we can apply Lemma \ref{chq} for $\ell=L$, $W=\iota(B)$, $V=C(K)$, $C=\Psd{Z}{K}$ with $Z:=\spn\{f,|f|: f\in\iota(B)\}$, so to obtain a $K$--positive extension $\tilde{\tilde{L}}$ of $\tilde{L}$ to $(\iota(B)-\Psd{Z}{K})\cap (\iota(B)-\Psd{Z}{K})$. Since for any $b\in B$ we have
} 
\[
\begin{array}{lcl}
	\pm|\hat{b}| & = & \rho_K(b)\hat{q} - (\rho_K(b)\hat{q}\mp|\hat{b}|)\\
		& = & -\rho_K(b)\hat{q} + (\rho_K(b)\hat{q}\pm|\hat{b}|)\\
		& \in & (\iota(B) - \Psd{Z}{K}) \cap (\iota(B) + \Psd{Z}{K}),
\end{array}
\]
{we easily get that $(\iota(B)-\Psd{Z}{K})\cap (\iota(B)-\Psd{Z}{K})=Z$.}
Now taking $S=\sum\Cnt{}{K}^2$ or $\Psd{\Cnt{}{K}}{K}$, and a positive $D\ge L(q)$ we have that, for all $b\in B$:
\[
\begin{array}{lcl}
	\tilde{L}(\hat{b}) = \tilde{\tilde{L}}(\hat{b}_+ - \hat{b}_-) \leq \tilde{\tilde{L}}(\hat{b}_+) \leq \tilde{\tilde{L}}(\rho_K(\hat{b}_+)\hat{q}) = \rho_K(\hat{b}_+)\tilde{L}(\hat{q}) \leq  D\norm{S; \rho_K}{\hat{b}}.
\end{array}
\]
Then, by Theorem \ref{posext} applied to $A=\Cnt{}{K}$, $\rho=\rho_K$ and $\tilde{L}: {\iota(B)} \to \reals$, there exists a representing measure $\mu$ for $\tilde{L}$ supported in $\Sp{\rho_K}{\Cnt{}{K}}=K$ (this equality holds since $K$ is compact; cf. {Remark~\ref{compactification}}). Hence, for all $b\in B$ we have
$$L(b)=\tilde{L}(\hat{b})=\int \hat{b} d\mu$$
i.e., $\mu$ is a $K$--representing measure for $L$.
\end{proof}

\begin{thm}\label{trnctmmnt} \ Let $A$ be a unital commutative $\reals$--algebra. Suppose $K\subseteq\X{A}$ is closed and non-compact, $B\subsetneq A$ is a linear subspace, and there exists $p\in A\setminus B$ such that $\hat{p}\ge1$ on $K$, $B_p:=\spn(B\cup\{p\})$ contains~$1$, $B_p$ generates $A$ and the following holds: 
\begin{equation} \label{p-existence2}
\textrm{for all } b \in B, \; \sup_{\alpha \in K} \left|\ddfrac{\hat{b}(\alpha)}{\hat{p}(\alpha)}\right| < \infty.      
\end{equation}

Let $\map{L}{B}{\reals}$ a $K$--positive linear functional, and assume that $L$ has a $K$--positive extension $\bar{L}$ to $B_p$. Then there exists a 
positive Radon measure $\nu$ whose support is contained in $K$ such that 
\begin{equation}
	L(b)=\int\hat{b}~d\nu,\quad\textrm{ for all } b\in B,  \label{intform}
\end{equation}
i.e., $\nu$ is a $K$--representing measure for $L$.
\end{thm}

Before we give the proof of Theorem \ref{trnctmmnt}, we present some related examples and remarks. 

\begin{rem} \label{rem29} 
It is not hard to prove that condition (\ref{p-existence2}) in Theorem \ref{trnctmmnt} admits the following equivalent form:
\begin{equation} \label{p-existence}
	\textrm{ for all } b\in B \; \textrm{there exists } \lambda > 0 \textrm{ such that } \lambda p\pm b\in\Psd{B_p}{K}.
\end{equation}
We will often switch from one form to the other, depending upon the circumstances. 
\end{rem}

In the sequel, we will use $\reals[\underline{X}]_d$ to denote the space of polynomials of degree at most $d$ in $\underline{X} \equiv (X_1,\ldots,X_n)$; when $n=1$, we will write $\reals[X]$ instead of $\reals[\underline{X}]$.

\begin{exm}
In Theorem \ref{trnctmmnt}, the assumption $p \in A \setminus B$ is needed, as the following example shows; that is, assuming only that $p \in A$ might lead to the wrong conclusions. For example, if $B:=\reals[X]_4$, $A:=\mathbb{R}[X]$, $K=\reals \cong \X{A}$, and if $L(a_0 + a_1X + a_2X^2 + a_3X^3 + a_4X^4) := a_0 + a_1 + a_2 + a_3 + 2a_4$, then $p(X):=X^4+1$ satisfies the assumptions of Theorem \ref{trnctmmnt} with the exception of $p \in A \setminus B$, but the functional $L$ which is $\reals$--positive does not admit a representing measure, as shown in \cite[Example 2.1]{Curto-Fialkow-2008}. \qed
\end{exm}

\begin{rem}
The assumption $\hat{p}\ge1$ on $K$ in Theorem \ref{trnctmmnt} is a matter of convenience. Indeed, since $A$ is unital it would suffice to assume that $\hat{p}\ge0$. For, if there exists $p\in A$ such that $\hat{p}\ge0$ and \eqref{p-existence} holds, then $p+1$ always satisfies the requirement of being at least $1$. On the other hand, note that if $\inf\hat{p}=0$ and \eqref{p-existence} holds, then 
\begin{enumerate}
	\item{
	there exists a net $\{x_{\alpha}\}\subseteq K$ such that $\lim\limits_{\alpha}\hat{p}(x_{\alpha})=0$;
	}
	\item{
	for all $b\in B$ there exists $\lambda\ge0$ such that $\lambda\hat{p}\pm\hat{b}\ge0$ on $K$ and hence 
	$\lim\limits_{\alpha}\hat{b}(x_{\alpha})=0$. Since $B_p$ generates $A$, this holds for all elements of $A$. 
	}
\end{enumerate}
Observe now that $\gamma(a):=\lim\limits_{\alpha}\hat{a}(x_{\alpha})$ is an $\reals$--algebra homomorphism which vanishes on $A$, which is impossible since $A$ is unital. Thus, the existence of a $p$ satisfying $\inf\hat{p}=0$ is inconsistent with \eqref{p-existence} when $A$ is unital. 
\end{rem}

We will now show that the requirement that the subspace $B_p$ generates the algebra $A$ cannot be removed from the hypotheses in Theorem \ref{trnctmmnt}. 

\begin{exm} \label{rem38}
Let $A:=\mathbb{R}[X,Y]$, $B:=\mathbb{R}[X]_3 \subseteq \mathbb{R}[X,Y]$, $K:=\mathbb{R} \times \{0\} \subseteq \mathbb{R}^2 \cong \X{A}$ and $p(X,Y):=1+X^4$. Clearly, $p \in A \setminus B$, $\hat{p} \ge 1$ on $K$, $B_p=\mathbb{R}[X]_4$ contains $1$, and for every $b \in B$, given as $b(X,0) \equiv b_0+b_1X+b_2X^2+b_3X^3$ one has 
$$
\left|\frac{b(x,0)}{p(x,0)}\right| \le 4 \max_{i=0,1,2,3}\left|b_i\right|\; \; \; \textrm{for all } (x,0) \in K.
$$
However, the subalgebra of $A$ generated by $B_p$ is $\mathbb{R}[X] \ne A$. Consider now the linear functional $L:B \rightarrow \mathbb{R}$ defined by $L(b):=b(1,0)$. It is straightforward to verify that $L$ is $K$--positive. Now, let $\bar{L}:B_p \rightarrow \mathbb{R}$ be given by $\bar{L}(c_0+c_1X+c_2X^2+c_3X^3+c_4X^4)=c_0+c_1+c_2+c_3+2c_4$. That is, $\bar{L}(c)=c(1,0)+c_4.$ \ Now observe that if $c \in B_p$ satisfies $c \ge 0$ on $K$, then $c_4 \ge 0$. It follows that $\bar{L}$ is a $K$--positive extension of $L$. However, $\bar{L}$ does not have a $K$--representing measure, as shown in \cite[Example 2.1]{Curto-Fialkow-2008}.  
\end{exm}

\begin{proof}[Proof of Theorem \ref{trnctmmnt}]
According to \eqref{p-existence}, for each $c\in B_p$, the function $\frac{\hat{c}}{\hat{p}}$ is bounded and 
continuous on $K$. So, we can consider the linear map:
\[
\begin{array}{cccl}
	\iota: &B_p&\longrightarrow &\Cnt{b}{K}\\
	 & c &\mapsto& \frac{\hat{c}}{\hat{p}}.
\end{array}
\]
By assumption, there exists a linear functional $\bar{L}:B_p\to \reals$ such that $\bar{L}(b)=L(b) \; (\textrm{for all } b\in B)$ and $\bar{L}(\Psd{B_p}{K})\subseteq[0,+\infty)$. Note that $\ker(\iota)\subseteq \ker(\bar{L})$. To see this, let $c\in\ker(\iota)$, which translates to $\iota c|_{K}=0$ and hence $\pm c\in\Psd{B_p}{K}$.
Thus $\bar{L}(\pm c)\ge0$, and hence $\bar{L}(c)=0$. Therefore, the following linear functional is well-defined:
\[
\begin{array}{cccl}
	\tilde{L}: &\iota(B_p)&\longrightarrow &\reals\\
	 & \frac{\hat{c}}{\hat{p}} &\mapsto& \bar{L}(c).
	\end{array}
\]
Note that $\Psd{\iota(B_p)}{K}=\{f\in\iota(B_p) : f\geq 0\ \text{on}\ K\}$. As a result, 
\begin{equation}\label{psd-iotaBp}
\tilde{L}(\Psd{\iota(B_p)}{K})=\tilde{L}(\iota\Psd{B_p}{K})=\bar{L}(\Psd{B_p}{K})\subseteq[0,+\infty).
\end{equation} 

Let us denote by $\B$ the algebra generated by the elements of $\iota(B_p)$ in $\Cnt{b}{K}$. \\
\textbf{Claim 1}. $\mathcal{B}$ separates points of $K$.\newline
\textit{Proof}. Take $x_1\neq x_2\in K$. Note that $B_p$ must separate elements of $K$ since $A$ separates $\X{A}$ (hence $K$) so, if for all $b\in B_p$ we have $b(x_1)=b(x_2)$, then this is also the case for all $a\in A$ (since $B_p$ generates $A$), which is impossible. Two cases arise: if $\hat{p}(x_1)=\hat{p}(x_2)$, then for some $b\in B$, $\hat{b}(x_1)\neq\hat{b}(x_2)$, and $\iota b(x_1)\neq\iota b(x_2)$; if, instead, $\hat{p}(x_1)\neq\hat{p}(x_2)$, then $\iota1(x_1)=\frac{1}{\hat{p}(x_1)}\neq\frac{1}{\hat{p}(x_2)}=\iota1(x_2)$.
 $\qed_{\textrm{Claim 1}}$

\medskip
Therefore, by Remark \ref{compactification}, $\tilde{K}:=\Sp{\rho_K}{\B}$ is a Hausdorff compactification of $K$ and hence 
$\Psd{\B}{K}=\Psd{\B}{\tilde{K}}$. Thus, Theorem \ref{cmpttrnctmmnt} applies to $\tilde{L}: \iota(B_p)\subset\mathcal{B}\to\reals$ and $q=1=\iota(p)$, giving a positive Radon measure $\mu$ supported in $\tilde{K}$ such that
\[
	\tilde{L}\left(\frac{\hat{c}}{\hat{p}}\right)=\int\frac{\hat{c}}{\hat{p}}~d\mu,\quad \textrm{ for all } c\in B_p,
\]
where, with abuse of notation, we have still denoted by $\frac{\hat{c}}{\hat{p}}$ the continuous extension to $\tilde{K}$. The measure $\mu$ is a $\tilde{K}$--representing measure for $\tilde{L}$.  

Consider now the restriction of $\tilde{L}$ to $B$, given as
\[
L(b)=\bar{L}(b)=\tilde{L}\left(\frac{\hat{b}}{\hat{p}}\right)= \int\hat{b} \cdot \tilde{p}~d\mu,\quad \textrm{ for all } b\in B,
\]
{where $\tilde{p}$ denotes the continuous extension of $\frac{1}{\hat{p}}$}
to $\tilde{K}$ (see Remark \ref{compactification}). Now recall that $\hat{p}>0$ on $K$ (which implies $\hat{p} \ge 0$ on $\tilde{K}$), and set $\nu(d x):=\tilde{p}(x)\mu(dx)$. Clearly, $\nu$ is a positive measure on $\tilde{K}$, and we have
\[
	L(b)=\bar{L}(b)= \int\hat{b} \cdot \tilde{p}~d\mu = \int\hat{b}~d\nu,\quad \textrm{ for all } b\in B,
\]
which yields in particular that $\nu$ is a $\tilde{K}$--representing measure for $L$.
	
It remains to show that $\nu$ is supported in $K$, and this will imply that $\nu$ is actually a $K$--representing measure for $L$. \\
\textbf{Claim 2}. $\tilde{p}(x)=0$ for all $x\in R=\tilde{K}\setminus K$.\\ 
{\textit{Proof}}. Let $x\in R$. Then there exists a net $\{x_{\alpha}\}_{\alpha}\subseteq K$ such that $x=\lim\limits_{\alpha}x_{\alpha}$ in $\tilde{K} \subseteq \X \B$. If $\tilde{p}(x)\neq 0$, then  by \eqref{p-existence} we have that $\lim\limits_{\alpha}\frac{\hat{b}(x_{\alpha})}{\hat{p}(x_{\alpha})}<\infty$ for all $b\in B$. Thus, $\hat{b}(x)<\infty$ for all $b\in B$. This together with the fact that $A$ is generated by $B_p$ implies that for all $a\in A$, 
$\hat{a}(x)<\infty$ and so we have that $x\in \X{A}$. Hence, since $K$ is a closed subset of $\X{A}$, we get that $x\in K$ which is a contradiction. 
This shows that $\tilde{p}$ is identically zero on $R$ and proves the claim. $\qed_{\textrm{Claim 2}}$

Let $M$ be a Borel measurable subset of $\tilde{K}$ such that $M\cap K=\emptyset$. Since $M \subseteq \tilde{K} \setminus K$, Claim 2 implies that $\nu(M)=0$. This proves that $\nu$ is supported in $K$.
\end{proof}

\begin{rem}
The reader may wonder whether the integral representation (\ref{intform}) for $L$ can be extended to $B_p$; i.e., if $\nu$ is a $K$--representing measure for $\bar{L}$. Since $B_p$ is the span of $B$ and $p$, this would amount to extending the representation to $p$. Now, 
\begin{eqnarray*}
\bar{L}(p)=\tilde{L}(1)&=&\mu(\tilde{K})=\mu(\tilde{K} \setminus K)+\mu(K) \\
&=&\mu(\tilde{K} \setminus K)+\int_K \hat{p} \cdot \frac{1}{\hat{p}}~d\mu=\mu(\tilde{K} \setminus K)+\int_K \hat{p}~d\nu.
\end{eqnarray*}
Thus, the possibility of extending (\ref{intform}) to $p$ depends on whether $\mu(\tilde{K} \setminus K)=0$. In the following example we will see that the above mentioned integral representation (with a measure $\nu$ supported in $K$) does not necessarily extend to $B_p$.
\end{rem}

\begin{exm}
Let $A:=\mathbb{R}[X]$, $B:=\mathbb{R}[X]_3$, $K:=\mathbb{R}$ and $p(X):=1+X^4$. Thus, $B_p=\mathbb{R}[X]_4$. Consider now the linear functional on $B$ given by evaluation at the point $1$; that is, $L(b):=b(1)$. It is clear that $L$ admits an integral representation, given by $\nu:=\delta_{1}$, the Dirac delta at $1$. We now define a new linear functional $\bar{L}$ on $B_p$, by $\bar{L}(c):=c(1)+c_4$, where $c_4$ denotes the coefficient of $c$ in $X^4$. Observe that $\bar{L}$ extends $L$ (since $c_4=0$ whenever $c \in \mathbb{R}[X]_3$), and that $\bar{L}$ is $K$--positive; for, if $c \in \mathbb{R}[X]_4$ and $c(x) \le 0$ in $\mathbb{R}$, then $c_4 \ge 0$, and consequently $\bar{L}(c) \ge 0$. If $\bar{L}$ admitted an integral representation over $K$, with measure $\mu$, then $0=L(X^2-2X+1)=L((X-1)^2)=\bar{L}((X-1)^2)=\int_K (x-1)^2~d\mu$ would imply that $x=1$ inside the support of $\mu$; that is, $\mu \restriction_K=\delta_{1}$. On the other hand, the polynomial $q(X):=X^4-X^3$ satisfies $\bar{L}(q)=1 \ne 0 = \int_K q~d\mu$. This contradiction shows that, in general, one cannot expect that (\ref{intform}) extends to $\bar{L}$. In the case at hand, the explanation lies in the support of $\mu$, which in this case extends beyond $K$ into $\tilde{K}=K \cup \{\pm \infty\}$, and therefore $\mu(\tilde{K} \setminus K)=1>0$. \qed
\end{exm}  

We now present a special instance of Theorem \ref{trnctmmnt} of independent interest. It deals with the case of a closed subset $K$ of the character space $\X{A}$ of the form $K=K_1 \times K_2$, with $K_1$ compact and $K_2$ closed; that is, $K$ is of cylindrical shape with compact base. We will focus on the specific case of polynomials in two variables $X \in \mathbb{R}^m$ and $Y \in \mathbb{R}^n$. Concretely, we have the following application of Theorem \ref{trnctmmnt}.

\begin{thm}\label{thm-partially-cmpct}
Let $A:=\mathbb{R}[X,Y]$ with $X \equiv (X_1,\ldots,X_m)$ and $Y \equiv (Y_1,\ldots,Y_n)$, and let $K:=K_1 \times \mathbb{R}^n \subseteq \mathbb{R}^m \times \mathbb{R}^n \cong \X{A}$, with $K_1$ compact in $\mathbb{R}^m$. Also, let $d \in \naturals$, $B:=(\mathbb{R}[X])[Y]_{2d-1}$, and let $p(X,Y)$ be a nonnegative polynomial of degree $2d$ in $Y$. Assume that $L:B \rightarrow \mathbb{R}$ is a $K$--positive linear functional, with a $K$--positive extension to $B_p$. Then there exists a positive Radon measure $\nu$ which is a $K$--representing measure for $L$.
\end{thm}

\begin{proof}
Let $q(X,Y):=1+p(X,Y)$. Given a polynomial $b \in B$, written as $b(X,Y)=b_0(X) + b_1(X)Y + \cdots + b_{2d-1}(X)Y^{2d-1}$, let $C_i:=\sup_{x \in K_1} \left|b_i(X)\right| \newline (i=0,1,\cdots,2d-1)$ and let $C:=\max \{C_0,C_1,\cdots,C_{2d-1}\}$. For every $b \in B$ and $(x,y) \in K_1 \times \mathbb{R}^n$, we have 
$$
\left|\frac{b(x,y)}{q(x,y)}\right| \le \frac{C_0+C_1 \left|y\right|+\cdots+C_{2d-1}\left|y\right|^{2d-1}}{q(x,y)}.
$$
Since $q$ has degree $2d$ in $Y$, and is bounded below by $1$, it is clear that we can find a positive constant $\lambda$ (which depends upon $b$) such that $\left|\frac{b(x,y)}{q(x,y)}\right| \le \lambda$, as needed for (\ref{p-existence}). The result now follows from Theorem \ref{trnctmmnt}.
\end{proof}

We conclude this section with a brief discussion of a link between Theorems \ref{cmpttrnctmmnt} and \ref{trnctmmnt} and the following generalization due to M. Marshall (cf. \cite[3.2.2. Theorem]{MurrayBook}) of the classical Haviland Theorem \cite{Hav}. We reformulate it in our setting as follows.

\begin{thm} \label{newcorRH} 
(Generalized Haviland Theorem) \ Let $A$ be a unital commutative $\mathbb{R}$--algebra and let $K$ be a closed subset of $\mathcal{X}(A)$. Suppose there exists $p \in A$ such that $\hat{p} \ge 0$ on $K$ and, for each integer $i \ge 1$, the set $K_i:=\{x \in K: \hat{p}(x) \le i\}$ is compact. Then for any $K$--positive linear functional $L:A \rightarrow\mathbb{R}$ there exists a positive Radon measure $\mu$ supported in $K$ such that $L(a)=\int \hat{a} \; d\mu \textrm{for all } a \in A$.
\end{thm}

Applying Theorem \ref{newcorRH} in the special case of $A:=\reals[\underline{X}]$ and $p(X):=X_1^2 + \ldots + X_n^2$ to a closed subset $K$ of $\reals^n$, with $K_i$ compact for each integer $i \ge 1$, Marshall obtained the classical Haviland Theorem.

The element $p$ in Theorem \ref{newcorRH} creates an ascending chain of compact subsets of $K$; that is, $K=\bigcup_{i=1}^{\infty} K_i$. As a result, $K$ is $\sigma$--compact. When $\hat{p}$ is bounded (as a continuous function on $K$), say $\hat{p} \le M$ for some positive integer $M$, it is easy to see that $K$ is compact; for, $K=K_M$ in that case. We can then appeal to Theorem \ref{cmpttrnctmmnt} to find the positive Radon measure $\mu$ that represents $L$. However, in the general case of $K$ closed, we have not been able to apply Theorem \ref{trnctmmnt} to obtain a new proof of Theorem \ref{newcorRH}. We believe this should be possible.

\section{\sc{\large{Support of the Representing Measures}}}\label{SupportRepMeas}

Our main results (Theorems \ref{cmpttrnctmmnt} and \ref{trnctmmnt}) deal with the existence of a representing measure with prescribed support of a positive linear functional on linear subspaces of $A$. In practice it is usually desirable to have some more information about the support of the 
representing measure. In fact, when dealing with a finite dimensional linear subspace $B$ of $\rx$, V. Tchakaloff established that the 
support of the representing measure is finite and its cardinality is bounded by the dimension of $B$ (see \cite[Theorem II]{Tchakaloff}, \cite[Theorem 2]{BT} and \cite{Richter}). The main tool enabling us to provide, in this case, a more precise description for the support is 
Carath\'eodory's Theorem. 

\subsection{\sc A generalization of Tchakaloff's Theorem} \ We aim to provide the following generalization of the celebrated theorem of Tchakaloff \cite[Theorem II]{Tchakaloff}. 

\begin{thm} \label{generalizedTchakaloff}
Let $K$ be compact and $B$ a linear subspace of $C(K)$, with $\dim B=N<\infty$. Let $L:B \rightarrow\reals$ be a linear functional such that $L(1)=1=\norm{}{L}$ and $L(b)\geq 0$ for all $b\in B$. Then there exist $x_1,\dots,x_m \in K$ and 
$\lambda_1,\dots,\lambda_m\ge0$, $m\leq N$, with $\sum_{i=1}^m\lambda_i=1$ such that 
\[
	\mu := \sum_{i=1}^m\lambda_i\delta_{x_i},
\]
is a $K$-representing measure for $L$, where $\delta_x$ is the Dirac measure at $x$.
\end{thm}

Theorem \ref{generalizedTchakaloff} can be derived from the Choquet-Bishop-de Leeuw Theorem, as we now discuss.

The \textit{state space} of $B$ 
is the subset of the dual $B^{\ast}$ of $B$ given by
\[
	\St{B}:=\{L\in B^{\ast}: L(1)=1=\norm{}{L}, \; L(b)\geq 0\, \textrm{ for all } b\in B \}.
\]
$\St{B}$ is a closed and convex subset of the unit ball in $B^{\ast}$ and hence it is compact in the weak-$\ast$ topology. The evaluation map $\map{\mathcal{E}}{K}{\St{B}}$, which sends every $x\in K$ to $e_x$ (the evaluation at $x$), is a topological embedding. \\ By the Krein-Milman Theorem, 
\begin{equation} \label{KM}
\St{B}=\overline{\textrm{co}}(\textrm{ext}(\St{B})),
\end{equation}
where $\overline{\textrm{co}}(\textrm{ext}(\St{B}))$ denotes the closure of the convex hull of the extreme points of $\St{B}$. We briefly pause to give a different description of $\St{B}$.

\begin{lemma} \label{newlem} 
For $K$ compact and $B$ a linear subspace of $C(K)$, we have
$$
\St{B}=\overline{\textrm{co}}(\mathcal{E}(K)).
$$
\end{lemma}

\begin{proof}
Since $\mathcal{E}(K) \subseteq \St{B}$, and $\St{B}$ is a compact and convex subset of the unit ball in $B^{\ast}$ with respect to the weak-$\ast$ topology, it is clear that $C:=\overline{\textrm{co}}(\mathcal{E}(K)) \subseteq \St{B}$. For the other inclusion, recall first that since $B$ is a locally convex topological vector space, we know that $B$ is isometrically isomorphic to the dual of its dual equipped with the weak-$\ast$ topology; in symbols, 
\begin{equation} \label{dual}
(B^{\ast},w^{\ast})^{\ast}\cong B
\end{equation}
(cf. \cite[Theorem 8.1.4]{Con2}). Suppose now that there exists a linear functional $L_0 \in \St{B}$ such that $L_0 \notin C $. The singleton $\{L_0\}$ is a weak-$\ast$ compact set satisfying $\{L_0\} \bigcap C = \emptyset$. It follows that there exists $\varphi \in (B^{\ast},w^{\ast})^{\ast}$ and a real number $\alpha$ such that $\varphi(L_0)<\alpha$ and $\varphi(L) \ge \alpha$ for every $L \in C$ (cf. \cite[Corollary 2.2.3]{Edw}). By (\ref{dual}), there exists $f \in B$ such that $\varphi(L)=L(f)$ for all $L \in B^{\ast}$ and, in particular, for all $e_x \in \mathcal{E}(K)$. Therefore, $f(x)=L(f) \ge \alpha$ for all $x \in K$. We then have $f-\alpha \ge 0$ on $K$. Since $L_0 \in \St{B}$, we must have $L_0(f-\alpha) \ge 0$; that is, $L_0(f) \ge \alpha$, a contradiction.
\end{proof}

By Lemma \ref{newlem} and (\ref{KM}), we now see that $\textrm{ext}(\St{B})\subseteq\mathcal{E}(K)$. For, amongst all closed subsets $F$ of $\St{B}$ with the property that $\overline{\textrm{co}}(F)=\St{B}$, the set of extreme points is the smallest (cf. \cite[Theorem 8.4.6]{Con2}). 

The Choquet boundary 
of $K$ (with respect to $B$) is defined as
\[
	\textrm{Ch}_K(B):=\mathcal{E}^{-1}(\textrm{ext}(\St{B}))\subseteq K.
\]

The Choquet-Bishop-de Leeuw Theorem states that for every $L\in\St{B}$ there exists a measure (not necessarily Radon) supported in 
$\textrm{Ch}_K(B)$ such that
\[
	L(b)=\int\hat{b}~d\mu, \textrm{ for all } b\in B
\]
and $\mu(\textrm{Ch}_K(B))=1$. Then Theorem \ref{generalizedTchakaloff} is an immediate consequence of the Choquet-Bishop-de Leeuw Theorem and provides the desired generalization of Tchakaloff's theorem as well as its versions in \cite{tcmp8, Curto-Fialkow-2005, Curto-Fialkow-2008, Fialkow-Nie}. Thus, Lemma \ref{newlem} and the subsequent discussion can be regarded as a generalization of Tchakaloff's Theorem \cite[Theorem II]{Tchakaloff}. \\ On the other hand, it was proved in \cite{BT} (see also \cite{Richter}, \cite[Corollary 1.25]{Sch} and \cite[Theorem 6]{diDio}) that a soluble $B$--truncated moment problem, with $B$ finite dimensional linear subspace of $C(K)$ and $K$ locally compact, always admits a finitely atomic representing measure.

\subsection{Classes of measures} \ For the space of polynomials $\mathbb{R}[X]$, regarded as a linear subspace of $C(K)$ (with $K$ compact in $\reals$), we know that $\mathbb{R}[X]$ is dense in $\mathcal{L}^1(\mu)$ if and only if $\mu$ is an extremal point of the unit ball of the space of measures on $K$; that is, $\mu$ is a Dirac delta. The result below attempts to mimic this when $\mathbb{R}[X]$ is replaced by an arbitrary linear subspace $B$ of $A$. In other words, think of $B$ as providing a generalization of the space of polynomials. This naturally leads to the notion of equivalence relation defined below.

Let $A=\Cnt{}{K}$, with $K$ compact. It is well known that every state $L\in\St{A}$ is representable via a positive Radon measure $\mu$ supported in $K$, i.e., $L(a)=L_{\mu}(a):=\int a~d\mu$ for all $a\in A$; this is the Riesz–Markov–Kakutani Theorem \cite[4.3.8. Theorem]{Con2}.

On the other hand, every linear subspace $B$ of $A$ containing $1$ (not necessarily finite dimensional) induces an equivalence relation on $\St{A}$, denoted by $\sim_B$ and defined by
\[
	L_1\sim_B L_2 \Leftrightarrow L_1|_B = L_2|_B.
\]

Note that if $L_1\sim_B L_2$, then for every $t\in[0,1]$ and $b\in B$, $L_1(b)=tL_1(b)+(1-t)L_2(b)= L_2(b)$, so 
equivalence classes of $\sim_B$ are convex. \\ Moreover, if $(L_{\alpha})_{\alpha}$ is a convergent net of 
$\sim_B$--equivalent states in weak-$\ast$ topology, then for every $b\in B$ and any $L\in[L_\alpha]_{\sim_B}$, 
$\lim\limits_{\alpha}L_{\alpha}(b)=L(b)$, so equivalence classes are also closed and hence compact. \\ 
By the Krein-Milman Theorem, $[L]_{\sim_B}$ is the closure of the convex hull of its extreme points.
The following result characterizes the extreme points of $\sim_B$ equivalence classes in terms of representing measures. \\ 
This was originally proved by Naimark \cite{Naimark} for the algebra of polynomials.
The proof given here is a modification of \cite[Theorem 2.3.4]{Akh} due to Gelfand. In fact, Gelfand proved this for $B=\mathbb{C}[z]$; here we adapt his proof to the case of a general $B$. 

\begin{thm} \label{thm42}
Let $K$ be compact and $B \subseteq C(K)$. The linear functional $L=L_{\mu}$ is an extreme point of $[L]_{\sim_B}$ if and only if 
$\cl{\norm{1,\mu}{}}{B}=\mathcal{L}^1(\mu)$.
\end{thm}
\begin{proof}
($\Rightarrow$) 
Suppose that $\cl{\norm{1,\mu}{}}{B}\neq\mathcal{L}^1(\mu)$. Let $B^{\perp}:=\{\ell\in\mathcal{L}^1(\mu)':\ell|_B=0\}$; $(B^{\perp})^{\perp}$ is defined accordingly.

Since $\cl{\norm{1,\mu}{}}{B}=(B^{\perp})^{\perp}$, we conclude that $B^{\perp}\neq \{0\}$. \\ Take $0\neq\ell\in B^{\perp}\subseteq\mathcal{L}^1(\mu)'=\mathcal{L}^{\infty}(\mu)$. There exists a non-constant $f\in\mathcal{L}^{\infty}(\mu)$ such that $\ell(g)=\int gf~d\mu$ for all $g\in\mathcal{L}^1(\mu)$. \\ Without loss of generality, assume $\norm{\infty}{f}\leq1$ and let $\nu_{\pm}=(1\pm f)\mu$. \\ Observe that $\nu_{\pm} \ge 0.$ Then $\mu=\frac{1}{2}(\nu_++\nu_-)$ and for every $b\in B$,
\[
\begin{array}{lcl}
	\int b~d\nu_{\pm} = \int b(1\pm f)~d\mu = \int b~d\mu\pm\int bf~d\mu = \int b~d\mu\pm\ell(b) = \int b~d\mu.
\end{array}
\]
Thus $L_{\nu_{\pm}}\in[L_{\mu}]_B$ and $L_{\mu}$ can not be an extreme point, a contradiction. \\ 
Hence $\cl{\norm{1,\mu}{}}{B}=\mathcal{L}^1(\mu)$.

($\Leftarrow$)
Suppose that $L_{\mu}$ is not an extreme point of $[L_{\mu}]_{\sim_B}$. \\ 
Then $\mu=\frac{1}{2}(\nu_1+\nu_2)$ for some $\nu_1 \ne \nu_2$ whose corresponding functionals 
$L_{\nu_1}, L_{\nu_2}\in[L_\mu]_{\sim_B}$, and as a consequence $\nu_i\leq2\mu$, $i=1,2$. \\ Therefore, $\nu_i\ll2\mu$ and by the Radon-Nikodym Theorem there exist positive measurable functions $f_1\neq f_2$ such that 
$\int g~d\nu_i=\int g f_i~2d\mu$ for all $g \in \mathcal{L}^1(\mu)$, $i=1,2$. \\ Since $L_{\nu_1},L_{\nu_2}\in[L_{\mu}]_{\sim_B}$ we get $\int b(f_1-f_2)~2d\mu=L_{\nu_{1}}(b)-L_{\nu_{2}}(b)=0$ 
for all $b\in B$. Thus, $0\neq(f_1-f_2)\in\mathcal{L}^{\infty}(\mu)\cap B^{\perp}$, which implies that 
$(B^{\perp})^{\perp}=\cl{\norm{1,\mu}{}}{B}\neq\mathcal{L}^1(\mu)$, a contradiction.
\end{proof}

\section{\sc{\large{Full Moment Problem VS Truncated Moment Problem}}}\label{FullVSTruncated}

In this section we prove an analogue of J. Stochel's theorem \cite{JSt} about the connection between full and truncated moment problems in the general setting of unital commutative $\reals$--algebras (see Definition \ref{B-TKMP}).  Stochel proved that the full moment problem for a linear functional $\map{L}{\rx}{\reals}$ has a solution if and only if for each $k \in \naturals$ the truncated moment problem for $L \restriction \rx_k$ has a solution. Recall that $\reals[\ux]_k$ denotes the space of all polynomials in $\reals[\ux]$ of degree at most $k$.

The following gives a version of Stochel's theorem in the more general framework of unital commutative $\reals$--algebras (see Theorem \ref{Stochel}).
\begin{dfn} \label{def51}
Let $\Omega$ be a directed set. A family $\{B_{\omega}\}_{\omega \in \Omega}$ is said to be \textit{subcofinal} if for any infinite directed $\Gamma \subseteq \Omega$, $\{B_{\omega}\}_{\omega \in \Gamma}$ is cofinal in $\{B_{\omega}\}_{\omega \in \Omega}$, i.e., for every $\omega\in\Omega$ there exists $\gamma\in\Gamma$ such that $\gamma\geq\omega$ and $B_{\gamma} \supseteq B_{\omega}$.

Let $K\subseteq\X{A}$ be a locally compact closed set.
A subcofinal family $\{B_{\omega}\}_{\omega\in\Omega}$ of linear subspaces of $A$ is called a \textit{truncated $K$--frame} if 
\begin{enumerate}
      \item{$\lim_{\omega\in\Omega} B_{\omega}=A$, where $\lim_{\omega\in\Omega} B_{\omega}:=\bigcup_{\omega}\bigcap_{\gamma \ge \omega}B_{\gamma}$,}
      \item{$1\in B_{\omega}$ for all $\omega \in \Omega$, and}
      \item{for every $f\in B_{\omega}$, there exist $\gamma\ge\omega$ and $p_f\in \Psd{B_{\gamma}}{K}$ such that 
      $\frac{\hat{f}}{\hat{p}_f}\in\Cnt{0}{K}$,}
\end{enumerate}
where $\Cnt{0}{K}$ denotes the algebra of continuous real-valued functions on $K$ that vanish at infinity. Since $\Cnt{0}{K}
\subseteq C_b(K)$, we equip $\Cnt{0}{K}$ with the norm $\rho_K$ in \eqref{sup-norm-K}, making it a Banach algebra. 
\end{dfn}

It is easy to show that a subnet of a $K-$frame is a $K-$frame. For $K \subseteq \X{A}$ closed, we denote by $\M{}{K}$ the space of all signed Radon measures supported in $K$ and by $\M{+}{K}$ the cone of positive Radon measures supported in $K$. We endow $\M{}{K}$ with the weak-$\ast$ topology. \label{pagemeasure}

\begin{thm}\label{Stochel}
Let $K$ be a locally compact closed subset of $\X{A}$ and $L$ a linear functional on $A$ such that $L(1)>0$. The full $K$--moment problem for $\map{L}{A}{\reals}$ has a solution if and only if there exists a truncated $K$--frame 
$\{B_{\omega}\}_{\omega\in\Omega}$ such that every $L_{\omega}=L|_{B_{\omega}}$ has a $K$--representing measure.
\end{thm}

\begin{proof}
($\Rightarrow$) is trivial, so we only prove ($\Leftarrow$).
Suppose that there exists a truncated $K$--frame 
$\{B_{\omega}\}_{\omega\in\Omega}$ such that for each $\omega\in\Omega$ there exists a measure $\mu_{\omega} \in \M{+}{K}$ with
$L_{\omega}(a)=\int\hat{a}~d\mu_{\omega} \textrm{ for all }$ $a\in B_\omega$. Therefore, $|\int {f}~d\mu_{\omega}|\leq\rho_{K}(f)\int 1~d\mu_{\omega}=L(1)\rho_{K}(f)$ for every $f\in\Cnt{0}{K}$, which implies that $\{\mu_{\omega}\}_{\omega\in\Omega}$ lies in the polar $\mathring{U}$ of $U:=\{f\in\Cnt{0}{K}: \rho_K(f)\leq \frac{1}{L(1)}\}$. Since the Banach-Alaoglu-Bourbaki theorem ensures the compactness of $\mathring{U}$ in $\M{+}{K}$ endowed with the weak-$\ast$ topology, there exists a subnet $\{\mu_{\gamma}\}_{\gamma\in\Gamma}$ of $\{\mu_{\omega}\}_{\omega\in\Omega}$ which converges to a measure $\mu_{\Gamma}$ in $\M{+}{K}$ in the weak-$\ast$ topology. 

Let $g\in \Psd{A}{K}$. By the inner regularity of $\mu_{\Gamma}$, we have that 
\begin{eqnarray*}
      \int_K \hat{g}~d\mu_{\Gamma} &=& \limsup\limits_{C\subseteq K\ \text{compact}}\int_C \hat{g}~d\mu_{\Gamma} \\
 &=& \limsup\limits_{C\subseteq K \ \text{compact}}\lim\limits_{\gamma\in\Gamma}\int_C \hat{g}~d\mu_{\gamma} \leq  \sup\limits_{\gamma\in\Gamma}\int_K \hat{g}~d\mu_{\gamma},
\end{eqnarray*}
which is finite, since $\left\{\int_K\hat{g}d\mu_\gamma\right\}_{\gamma\in\Gamma}$ is a subset of the compact set $\phi(\mathring{U})$, where $\phi$ is the continuous functional on $\M{+}{K}$ defined by $\phi(\nu):=\int_K\hat{g}d\nu$. Thus, $\hat{g}\mu_{\Gamma}\in\M{+}{K}$. Moreover, since $\Cnt{0}{K}$ is an ideal in 
$\Cnt{}{K}$, for every $f\in\Cnt{0}{K}$ we get $f\hat{g}\in\Cnt{0}{K}$ and hence 
$\lim\limits_{\gamma \in \Gamma}\int f\hat{g}~d\mu_{\gamma}=\int f\hat{g}~d\mu_{\Gamma}$,i.e., $\lim\limits_{\gamma \in \Gamma}\hat{g}\mu_{\gamma}=\hat{g}\mu_{\Gamma}$ in the weak-$\ast$ topology.

Now, $\{B_{\gamma}\}_{\gamma\in\Gamma}$ is a $K-$frame as a subnet of the $K-$frame $\{B_{\omega}\}_{\omega\in\Omega}$. Therefore, by \emph{1} in Definition~\ref{def51} we have that for every $q\in A$ there exists $\gamma_0\in\Gamma$ such that $q\in B_{\gamma}$ for all $\gamma\ge\gamma_0$ in $\Gamma$. Also, by \emph{3} in Definition~\ref{def51}, there exist $\delta\in\Gamma$ with $\delta\geq\gamma_0$ and $p_q \in \Psd{B_{\delta}}{K}$ such that $\ddfrac{\hat{q}}{\hat{p_q}} \in \Cnt{0}{K}$. Hence, 
$$
      L(q)  \!=\! \lim\limits_{\gamma\ge\gamma_0}L_{\gamma}(q) \!=\! \lim\limits_{\gamma\ge\gamma_0}\int\hat{q}~d\mu_{\gamma} \!=\! \lim\limits_{\gamma\ge\gamma_0}\int\frac{\hat{q}}{\hat{p_q}}\hat{p_q}~d\mu_{\gamma}\!=\!
						 \int\ddfrac{\hat{q}}{\hat{p}_q}\hat{p}_q~d\mu_{\Gamma} \!=\! \int \hat{q}~d\mu_{\Gamma},
$$
 which proves that $\mu_{\Gamma}$ represents $L$ as an integral over $K$.
\end{proof}
\begin{rem}
It is clear that for any closed $K\subseteq\reals^n$ the sequence $\{\rx_d\}_{d\in\naturals}$ is a truncated $K$--frame in $\rx$. Indeed, the subcofinality of $\{\rx_d\}_{d\in\naturals}$ follows from the fact that any infinite subset of $\naturals$ is cofinal in it, (1) and (2) in Definition \ref{def51} are obviously verified, and (3) is easily fulfilled by taking $p_f(\underline{X}):=1+\underline{X}^2+f^{2}(\underline{X})$. 
Thus, Stochel's Theorem \cite[Theorem 4]{JSt} follows from Theorem \ref{Stochel}.
\end{rem}

We now establish a link with a result of M. Putinar and F.-H. Vasilescu, dealing with full moment problems in $\mathbb{R}^n$. More precisely, we apply Theorems \ref{trnctmmnt} and \ref{Stochel} to obtain a new proof of sufficiency in the main result in \cite{PuVa}.  \ The setup is as follows. Let $\underline{g} \equiv \{g_1,\cdots,g_s\}$ be a finite collection of polynomials in $\mathbb{R}[\underline{X}]$, and define the rational function $\theta_{\underline{g}}$ by 
\begin{equation*} \label{eqPuVa}
\theta_{\underline{g}}(\underline{x}):=
\frac{1}{1+x_1^2+\cdots+x_n^2+g_1(\underline{x})^2+\cdots+g_s(\underline{x})^2} \quad (\underline{x} \in \mathbb{R}^n).
\end{equation*}
Moreover, denote by $\mathcal{R}_{\theta_{\underline{g}}}$ the $\mathbb{R}$--algebra generated by $\mathbb{R}[\underline{X}]$ and 
$\theta_{\underline{g}}$. Observe that $(\theta_{\underline{g}})^{-1}$ is a polynomial of degree $d:=\max \{2,2 \deg g_1,\ldots, 2 \deg g_s \}$. 

Before we state and prove the Putinar-Vasilescu result, we briefly recall a criterion for a rational function $f \in \mathcal{R}_{\theta_{\underline{g}}}$ to be nonnegative on $K$.

\begin{lemma} \label{thm623} (\cite[Theorem 6.2.3]{MurrayBook}) \ Let $S$ be a finite subset in $\rx$, and $T$ the preorder $K_S$ generated by $S$. Suppose $p \in \rx$, $p \ne 0$, $p-1 \in T$, and there exist integers $k,\ell \ge 0$ such that $kp^{\ell}-\Sigma_{i=1}^n X_i^2 \ge 0$. Then, for any $f \in \rx[\frac{1}{p}]:=\{\frac{f}{p^k}: f \in \rx, k \in \naturals_0 \}$, the following statements are equivalent:
\begin{enumerate}
\item[(1)] \ $f \ge 0$ on $K$; 
\item[(2)] \ there exists an integer $m$ such that for all real $\epsilon > 0$, $f+\epsilon p^m \in T[\frac{1}{p}]:=\{\frac{t}{p^k}:t \in T, \ k \in \naturals_0 \}$ .
\end{enumerate}
\end{lemma}
 
\begin{thm} (cf. \cite[Theorem 2.5]{PuVa}) \label{PuVaRevisited} \ Let $\theta_{\underline{g}}$ and $\mathcal{R}_{\theta_{\underline{g}}}$ be defined as above, and let $\Lambda:\mathcal{R}_{\theta_{\underline{g}}} \rightarrow \mathbb{R}$ be a linear map such that $\Lambda$ satisfies $\Lambda(r^2) \ge 0$ and $\Lambda(g_kr^2) \ge 0$, whenever $r \in \mathcal{R}_{\theta_{\underline{g}}}$ and $k=1,\cdots,s$. Then $\Lambda$ has a positive Radon representing measure, whose support is in the semialgebraic set $\bigcap_{k=1}^s g_k^{-1}(\mathbb{R}_+)$.
\end{thm}

\begin{proof}
Recall the notation and terminology used in Theorem \ref{trnctmmnt}. We let $A:=\mathcal{R}_{\theta_{\underline{g}}}$ and $K:= \bigcap_{k=1}^s g_k^{-1}(\mathbb{R}_+)$. Observe that $A$ is a unital commutative $\mathbb{R}$--algebra, and that $K \subseteq \X A$ is closed. For every $N \in \naturals_0$, we let 
$$
B_N:=\spn \{q_{\ell} (\theta_{\underline{g}})^{\ell}: \ell \in \naturals_0 \textrm{ and } \deg q_{\ell} \le d(\ell+N+1) \}.
$$
Observe that $B_N$ is a linear subspace of $A$, that $B_0 \subseteq B_1 \subseteq \cdots \subseteq B_N$, and that $1 \in B_0$. Moreover, $X_1,\ldots,X_n \in B_0$, and $\theta_{\underline{g}} \in B_0$. It follows that the algebra generated by $B_N$ is $A$ and $A$ is the increasing union of the linear subspaces $B_N$. In Figure \ref{Figure 0}, we give a visual representation of $B_0$, $B_1$ and $B_N$.

\psset{unit =.85mm, linewidth=.7\pslinewidth} 
\begin{figure}[h]
\begin{center}
\begin{pspicture}(165,65)

\rput(-6,22){

\psline(10,20)(75,20)
\psline(10,35)(75,35)
\psline(10,50)(75,50)
\psline(10,65)(75,65)
\psline(10,20)(10,75)
\psline(25,20)(25,75)
\psline(40,20)(40,75)
\psline(55,20)(55,75)
\psline(70,20)(70,75)

\put(35,9){$\deg q_{\ell}$}
\put(9,14){\footnotesize{$0$}}
\put(23,14){\footnotesize{$d$}}
\put(38,14){\footnotesize{$2d$}}
\put(53,14){\footnotesize{$3d$}}
\put(68,14){\footnotesize{$4d$}}

\put(3,20){\footnotesize{$1$}}
\put(2,35){\footnotesize{$\theta_{\underline{g}}$}}
\put(0,50){\footnotesize{$(\theta_{\underline{g}})^2$}}
\put(0,65){\footnotesize{$(\theta_{\underline{g}})^3$}}

\psframe[fillstyle=solid,fillcolor=black!10](8,18)(12,22)
\pscircle[fillstyle=solid,fillcolor=black](10,20){1.2}
 
\psframe[fillstyle=solid,fillcolor=black!10](23,18)(27,22)
\pscircle[fillstyle=solid,fillcolor=black](25,20){1.2}

\psframe[fillstyle=solid,fillcolor=black!10](38,18)(42,22)

\psframe[fillstyle=solid,fillcolor=black!10](8,33)(12,37)
\pscircle[fillstyle=solid,fillcolor=black](10,35){1.2}
 
\psframe[fillstyle=solid,fillcolor=black!10](23,33)(27,37)
\pscircle[fillstyle=solid,fillcolor=black](25,35){1.2}

\psframe[fillstyle=solid,fillcolor=black!10](38,33)(42,37)
\pscircle[fillstyle=solid,fillcolor=black](40,35){1.2}

\psframe[fillstyle=solid,fillcolor=black!10](53,33)(57,37)

\psframe[fillstyle=solid,fillcolor=black!10](8,48)(12,52)
\pscircle[fillstyle=solid,fillcolor=black](10,50){1.2}
 
\psframe[fillstyle=solid,fillcolor=black!10](23,48)(27,52)
\pscircle[fillstyle=solid,fillcolor=black](25,50){1.2}

\psframe[fillstyle=solid,fillcolor=black!10](38,48)(42,52)
\pscircle[fillstyle=solid,fillcolor=black](40,50){1.2}

\psframe[fillstyle=solid,fillcolor=black!10](53,48)(57,52)
\pscircle[fillstyle=solid,fillcolor=black](55,50){1.2}

\psframe[fillstyle=solid,fillcolor=black!10](68,48)(72,52)

\psframe[fillstyle=solid,fillcolor=black!10](8,63)(12,67)
\pscircle[fillstyle=solid,fillcolor=black](10,65){1.2}
 
\psframe[fillstyle=solid,fillcolor=black!10](23,63)(27,67)
\pscircle[fillstyle=solid,fillcolor=black](25,65){1.2}

\psframe[fillstyle=solid,fillcolor=black!10](38,63)(42,67)
\pscircle[fillstyle=solid,fillcolor=black](40,65){1.2}

\psframe[fillstyle=solid,fillcolor=black!10](53,63)(57,67)
\pscircle[fillstyle=solid,fillcolor=black](55,65){1.2}

\psframe[fillstyle=solid,fillcolor=black!10](68,63)(72,67)
\pscircle[fillstyle=solid,fillcolor=black](70,65){1.2}

\put(28,42){\Large{${B_0} \subseteq {B_1}$}}
}


\rput(-5,22){

\psline(90,20)(155,20)
\psline(90,35)(155,35)
\psline(90,50)(155,50)
\psline(90,65)(155,65)
\psline(90,20)(90,75)
\psline(105,20)(105,75)
\psline(120,20)(120,75)
\psline(135,20)(135,75)
\psline(150,20)(150,75)

\put(115,9){$\deg q_{\ell}$}
\put(89,14){\footnotesize{$0$}}
\put(102.7,14){$\cdots$}
\put(110,14){\footnotesize{$d(N+1)$}}
\put(126,14){\footnotesize{$d(N+2)$}}
\put(142,14){\footnotesize{$d(N+3)$}}

\put(83,20){\footnotesize{$1$}}
\put(82,35){\footnotesize{$\theta_{\underline{g}}$}}
\put(80,50){\footnotesize{$(\theta_{\underline{g}})^2$}}
\put(80,65){\footnotesize{$(\theta_{\underline{g}})^3$}}

\psdot[dotsize=3,dotstyle=square*,dotangle=45](90,20)

\psdot[dotsize=3,dotstyle=square*,dotangle=45](105,20)

\psdot[dotsize=3,dotstyle=square*,dotangle=45](120,20)

\psdot[dotsize=3,dotstyle=square*,dotangle=45](90,35)

\psdot[dotsize=3,dotstyle=square*,dotangle=45](105,35)

\psdot[dotsize=3,dotstyle=square*,dotangle=45](120,35)

\psdot[dotsize=3,dotstyle=square*,dotangle=45](135,35)

\psdot[dotsize=3,dotstyle=square*,dotangle=45](90,50)

\psdot[dotsize=3,dotstyle=square*,dotangle=45](105,50)

\psdot[dotsize=3,dotstyle=square*,dotangle=45](120,50)

\psdot[dotsize=3,dotstyle=square*,dotangle=45](135,50)

\psdot[dotsize=3,dotstyle=square*,dotangle=45](150,50)

\psdot[dotsize=3,dotstyle=square*,dotangle=45](90,65)

\psdot[dotsize=3,dotstyle=square*,dotangle=45](105,65)

\psdot[dotsize=3,dotstyle=square*,dotangle=45](120,65)

\psdot[dotsize=3,dotstyle=square*,dotangle=45](135,65)

\psdot[dotsize=3,dotstyle=square*,dotangle=45](150,65)

\put(110,42){\Large{$B_N$}}

}

\end{pspicture}
\end{center}
\caption{\small{On the left, diagram of $B_0$ (dots) and $B_1$ (squares); on the right, diagram of $B_N$ (diamonds)}}
\label{Figure 0}
\end{figure}

{
We now let $L_N:=\Lambda \restriction B_N$ and establish that is $K-$positive. Observe that $\rx[\frac{1}{p}]$ coincides with $\mathcal{R}_{\theta_{\underline{g}}}$ if we let $S:=\{g_1,\ldots,g_s\}$ and $p:=\frac{1}{\theta_{\underline{g}}}$. Let $f \in \mathcal{R}_{\theta_{\underline{g}}}$ and assume that $f \ge 0$ on $K$. By Lemma \ref{thm623}, there exists an integer $m$ such that for all $\epsilon > 0$, $f+\epsilon p^m \in T[\frac{1}{p}]$. 
It follows that $L_N(f+\epsilon p^m) = L_N(f) + \epsilon L_N(p^m) \ge 0$ for all $\epsilon > 0$. This readily implies that $L_N(f) \ge 0$.


If $K$ is compact, then we can immediately apply Theorem \ref{cmpttrnctmmnt} to each $L_N$ by taking $q=1$ and so obtain that there exists a $K-$representing measure $\nu_N$ for $L_N$ for all $N\in\naturals_0$.

If $K$ is non-compact, then we consider $p_{N+1}:=(\theta_{\underline{g}})^{-(N+2)}\in A \setminus B_N$, which is such that $\hat{p}_{N+1} \ge 1$ on $K$ and $\deg p_N = d(N+2)$, as we aim to apply Theorem \ref{trnctmmnt} to $B_N$, $A$, $K$, $p_{N+1}$. To do this, we need to ensure that condition (\ref{p-existence}) holds for every $b \in B_N$. As we know from Remark \ref{rem29}, it suffices to prove condition (\ref{p-existence2}). In turn, since $B_N$ is the span of elements of $A$ of the form $q_{\ell}(\theta_{\underline{g}})^{\ell}$,  with $\deg q_{\ell} \le d(\ell+N+1)$, we need to prove that quotients of the following form are bounded on $K$: 
\begin{eqnarray*}
\frac{\hat{q_{\ell}}(\hat{\theta_{\underline{g}}})^{\ell}}{\hat{p}_{N+1}}(\underline{x})\ &\equiv& 
\ddfrac{\frac{\hat{q_{\ell}}}{(1+x_1^2+\cdots+x_n^2+g_1(\underline{x})^2+\cdots+g_s(\underline{x})^2)^{\ell}}}{(1+x_1^2+\cdots+x_n^2+g_1(\underline{x})^2+\cdots+g_s(\underline{x})^2)^{N+2}} \\
&=&\ddfrac{\hat{q_{\ell}}}{(1+x_1^2+\cdots+x_n^2+g_1(\underline{x})^2+\cdots+g_s(\underline{x})^2)^{\ell+N+2}}.
\end{eqnarray*}}
But this is clear from the requirement on $\deg q_{\ell}$ for membership in $B_N$. Since $L_N$ admits a $K$--positive extension to $(B_N)_p$ (namely $L_{N+1}$), Theorem \ref{trnctmmnt} guarantees the existence of a positive Radon measure $\nu_N$ whose support is contained in $K$ and such that 
$$
L_N(b)= \int _K \hat{b} \; d\nu_N, \; \; \textrm{for all } b \in B_N.
$$
Having established the existence of a $K$--representing measure for $\Lambda$ on each $B_N$, we now want to apply Theorem \ref{Stochel} to the subcofinal family of linear subspaces $\{B_N\}_{N \in \naturals_0}$. For this, we need to prove that $\{B_N\}$ is a truncated $K$--frame. 
Conditions $1$ and $2$ in Definition \ref{def51} clearly hold; that is, $\lim_N B_N =\bigcup_N B_N= A$, and $1 \in B_N$ for all $N \in \naturals_0$. As for condition $3$, we easily verify that $\frac{\hat{b}}{\hat{p}_{N+1}} \in C_0(K)$, for all $b \in B_N$. Thus, all the hypotheses in Theorem \ref{Stochel} are satisfied (using the family of measures $\{\nu_N\}_{N=0}^{\infty}$), and therefore we obtain a $K$--representing measure $\nu$ supported in $K$.
\end{proof}

The approach we used to prove Theorem \ref{PuVaRevisited} works equally well to prove a result of M. Marshall on the localization of the multiplicative set of powers of a positive polynomial; we leave the details to the reader.

\begin{thm} \label{thm56} (cf. \cite[Corollary 6.2.4]{MurrayBook}) \ With the hypotheses as in Lemma \ref{thm623}, let $L:\rx[\frac{1}{p}] \rightarrow \mathbb{R}$ be a linear functional satisfying $L \ge 0$ on $T[\frac{1}{p}]$. Then there exists a positive Radon measure $\mu$ supported in $K$ such that $L(f)=\int_K f \; d\mu$, for all $f \in \rx[\frac{1}{p}]$ . 
\end{thm}


\section{\sc{\large{Applications}}} \label{Applications}

In this section we present a number of applications of our main results, ranging from the TMP for point processes, to the Classical TMP (including the so-called Triangular, Rectangular and Sparse Connected TMP), and to the Subnormal Completion Problem for $2$--variable weighted shifts.


\subsection{{The Truncated Moment Problem for Point Processes}} \label{KLSsection} Let $X$ be a Hausdorff locally compact whose topology has a countable basis (hence $X$ is $\sigma$--compact). Denote by $\Cnt{c}{X}$ the space of all continuous real-valued functions compactly supported in $X$. Using the notation on page~\pageref{pagemeasure} (right before Theorem \ref{Stochel}), let $\N X$ be the subset of all $\naturals_0$--valued measures in $\M{}{X}$, i.e.,
\begin{eqnarray*}
\N X&:=&\left\{\sum_{i\in I} \delta_{x_i}: x_i\in X,\ I\subseteq\naturals \text{ with either }  |I|<\infty \right.\\
& \ &\left.\text{ or } I=\naturals \text{ and } (x_i)_{i\in I} \text{ without accumulation points}\right\}
\end{eqnarray*}
We endow $\M{}{X}$ with the so-called vague topology $\tau$, i.e., the weakest topology on $\M{}{X}$ such the map $\M{}{X}\to\reals$, $\nu\mapsto \int_Xf d\mu$ is continuous for all $f\in\Cnt{c}{X}$. The space $\N X$ is a closed subset of $(\M{}{X}, \tau)$ (see \cite[Lemma~4.4]{Kall}). A \emph{point process} is a positive Radon probability measure on $\M{}{X}$ which is supported in $\N X$. In this section, we will study the problem of characterizing point processes in terms of their first few moments, namely, the truncated $\N X$--moment problem. This is a long-standing problem in the statistical physics literature, where it is often addressed as \emph{realizability problem} (see, e.g., \cite{P64}), and recently in \cite{KuLeSp11} solubility conditions of Haviland type were obtained. We will use Theorems \ref{cmpttrnctmmnt} and \ref{trnctmmnt} to provide solubility criteria for the truncated $\N X$--moment problem, retrieving some of the main results in \cite{KuLeSp11}. To this aim, let us first identify the algebra involved in the truncated $\N X$--moment problem.

For any $n\in\naturals$ and $\nu\in \M{}{X}$, we denote by $\nu^{\otimes n}$ the (symmetric) $n^{th}$ power of $\nu$, i.e., 
$$
\nu^{\otimes n}(dx_1, \ldots, dx_n):=\nu(dx_1)\cdots\nu(dx_n)
$$ 
and for any $f_n\in\Cnt{c}{X^n}$ we define
$$
f_n\nu^{\otimes n}:=\int_{X^n}f_n(x_1, \ldots, x_n)\nu^{\otimes n}(dx_1, \ldots, dx_n).
$$
We also set $f_0\nu^{\otimes 0}:=f_0$ for any $f_0\in\reals$.
Then it is clear that for any $n,m\in\naturals_0$, $f_n\in\Cnt{c}{X^n}$ and $g_m\in\Cnt{c}{X^m}$ we have
\begin{equation}\label{mult-cross}
(f_n\nu^{\otimes n})(g_m\nu^{\otimes m})=(f_n\otimes g_m)\nu^{\otimes (n+m)}.
\end{equation}

Let $\mathscr{P}$ be the space of all polynomials in the variable $\eta$ in $\M{}{X}$ and coefficients in $\Cnt{c}{X}$, i.e., $a\in \mathscr{P}$ is of the following form
$$a(\eta):=\sum_{j=0}^N f_j \eta^{\otimes j}, \ N\in\naturals_0,\  f_0\in\reals,\  f_j\in\Cnt{c}{X^j}.$$
The space $\mathscr{P}$ together with the multiplication defined in \eqref{mult-cross} is a unital commutative $\reals$--algebra.

\begin{prop}\label{prop-chara}
The space $(\M{}{X}, \tau)$ is topologically embedded in the character space $\X{\mathscr{P}}$ endowed with the weakest topology making all Gelfand transforms continuous.
\end{prop}

\begin{proof}
Let us consider the following map:
$$\begin{array}{llll}
\phi:&\M{}{X}&\to& \X{\mathscr{P}}\\
\ & \nu &\mapsto & \phi(\nu)(a):=\sum\limits_{j=0}^N \int_{X^j}f_j d\nu^{\otimes j}, \ \forall \ a(\eta):=\sum\limits_{j=0}^N f_j \eta^{\otimes j}\in \mathscr{P}.
\end{array}
$$
For any $\nu\in\M{}{X}$, $\phi(\nu)\in\X{\mathscr{P}}$ since for any $\lambda\in\reals$ and $a,b\in \mathscr{P}$, say $a(\eta):=\sum_{j=0}^N f_j \eta^{\otimes j}$ and $b(\eta):=\sum_{k=0}^M g_k \eta^{\otimes k}$, we hace $\phi(\nu)(ab)$=
\begin{eqnarray*}
& \!\!\!\!\!\!\!\!&\phi(\nu)\left(\sum_{j=0}^N\sum_{k=0}^M  (f_j\otimes g_k)\eta^{\otimes (j+k)}\right)=\sum_{j=0}^N\sum_{k=0}^M\int_{X^{j+k}}(f_j\otimes g_k)d\nu^{\otimes (j+k)}\\
&=&\left(\sum_{j=0}^N\int_{X^{j}}f_jd\nu^{\otimes j}\right)\left(\sum_{k=0}^M\int_{X^{k}}g_k d\nu^{\otimes k}\right)=\Bigg(\phi(\nu)(a)\Bigg)\Bigg(\phi(\nu)(b)\Bigg) 
\end{eqnarray*}
and
\begin{eqnarray*}
\phi(\nu)(a+\lambda b)&=&\phi(\nu)\left(\sum_{j=0}^N  f_j\eta^{\otimes j}+\lambda \sum_{k=0}^Mg_k\eta^{\otimes k}\right)\\
&=&\sum_{j=0}^N \int_{X^j} f_j d\nu^{\otimes j}+\lambda \sum_{k=0}^M\int_{X^k}g_kd\nu^{\otimes k}=\phi(\nu)(a)+\lambda\phi(\nu)(b).
\end{eqnarray*}
The map $\phi$ is injective as whenever $\nu, \nu'\in\M{}{X}$ satisfy $\int_X f d\nu=\int_X f d\nu'$ for all $f\in\Cnt{c}{X}$ we have $\nu\cong\nu'$. The continuity of $\phi$ directly follows from the definition of topologies on $\X{\mathscr{P}}$ and $\M{}{X}$. \
Moreover, it is easy to see that $\phi(\M{}{X})$ endowed with the topology induced by $\X{\mathscr{P}}$ is homeomorphic to $(\M{}{X}, \tau)$. 
\end{proof}

For $N \ge 1$, let us consider now the linear subspace $\mathcal{Q}$ of $\mathscr{P}$ consisting of all polynomials of degree at most $N$, i.e.
$$
\mathcal{Q}:=\left\{a\in \mathscr{P}: a(\eta)=f_0+\sum_{j=1}^N  f_j\eta^{\otimes j}, \ f_0\in\reals, f_j\in \Cnt{c}{X^j} \; (j=1,\ldots,N) \right\}.
$$

An immediate consequence of Theorem \ref{cmpttrnctmmnt} is the following

\begin{crl}\label{risultato1}
Let $K\subseteq \M{}{X}$ be compact and $L:\mathcal{Q}\to \reals$ linear. There exists a $K$--representing measure for $L$ if and only if $L$ is $K-$positive.
\end{crl}

\begin{proof}
Since $1\in\mathcal{Q}$, the conclusion follows by applying Theorem \ref{cmpttrnctmmnt} for $B:=\mathcal{Q}$ and $q(\eta):=1$.
\end{proof}

In \cite{KuLeSp11} the authors consider the factorial $n^{\textrm{th}}$ power $\eta^{\odot n}$ instead of the $n^{\textrm{th}}$ power $\eta^{\otimes n}$ we considered above. Denote by $\tilde{\mathscr{P}}$ the set of polynomials defined by replacing $\eta^{\otimes n}$ with  $\eta^{\odot n}$ in the above definition of $\mathscr{P}$. Then as sets $\mathscr{P}=\tilde{\mathscr{P}}$ and there is a bijective correspondence between $\N X$--positive linear functionals on $\mathscr{P}$ and $\N{X}$--positive linear functionals on $\tilde{\mathscr{P}}$. Hence, \cite[Theorem 3.4, Proposition 3.9 and Theorem 3.10]{KuLeSp11} directly follow from Corollary \ref{risultato1}. In \cite{KuLeSp11} the authors also study the case $K=\N{X}$ (see \cite[Theorem 3.14]{KuLeSp11} for the case when $X$ is compact and \cite[Theorem 3.17]{KuLeSp11} for $X$ non-compact). We plan to pursue the relation between these two results and Theorem \ref{trnctmmnt} in a future manuscript.




\subsection{{The Classical Truncated Moment Problem}} \ Theorem \ref{trnctmmnt} can be seen as a generalization of the Curto-Fialkow's solution of the Classical truncated moment 
problem \cite[Theorem 2.2]{Curto-Fialkow-2008} in two directions: \\
(1) it assumes a unital commutative $\reals$--algebra instead of $\rx$; \\
(2) it remains valid for a positive functional over any 
given linear subspace of $\rx$, without extra assumption of finite dimensionality. 

We show in Corollary \ref{poly-one-degree-higher} how to derive an improved version of \cite[Theorem 2.2]{Curto-Fialkow-2008} from Theorem~\ref{trnctmmnt}. To see this, we need the following lemma.

\begin{lemma}\label{min-bnd-poly}
For every monomial $\ux^{\alpha}$ of degree $2d$ or $2d+1$ with $d \in \naturals$, there exists a polynomial {$p_\alpha$} of degree at most $2d+2$ such that $|\ux^{\alpha}|\leq {p_\alpha}(\ux)$ and $p_\alpha(\underline{y})\geq 1$ for all $\underline{y}\in\reals^n$. 
\end{lemma}
\begin{proof}
Decompose $\alpha=\gamma+2\beta$, such that $\gamma=(\gamma_1,\cdots,\gamma_n)$ and $\gamma_i\in\{0, 1\}$ for each 
$i=1,\cdots,n$. So, $\ux^{\alpha}=\ux^{\gamma}\ux^{2\beta}$. \
If $\gamma\not\equiv 0$, then by arithmetic-geometric inequality,
\begin{eqnarray} \label{lemma61}
      |\ux^{\alpha}| = |\ux^{\gamma}|\ux^{2\beta} \leq \frac{1}{|\gamma|}\left(\sum_{i=1}^n\gamma_i|X_i|^{|\gamma|}\right)\ux^{2\beta}.
\end{eqnarray}
Observe now that, for any real $a$ and any integer $r$, we have $|a^r|\leq(1+a^2)^{\lceil\frac{r+1}{2}\rceil}$. Then
\[
\begin{array}{lcl}
      |\ux^{\alpha}| & \leq & \frac{1}{|\gamma|}\left(\sum_{i=1}^n\gamma_i(1+X_i^2)^{\lceil\frac{|\gamma|+1}{2}\rceil}\right)\ux^{2\beta}\\
       & \leq & \frac{1}{|\gamma|}\left(\sum_{i=1}^n\gamma_i(1+X_i^2)^{\lceil\frac{|\gamma|+1}{2}\rceil}\right)\prod_{i=1}^n(1+X_i^2)^{\beta_i}=:p_\alpha(\ux).
\end{array}
\]
Clearly $p_\alpha(\underline{y})\geq 1$ for all $\underline{y}\in\reals^n$. Moreover:\\
- \ If $|\alpha|=2d$, then $|\gamma|=2l$ with $l\ge 1$. Thus, $\deg(p_\alpha)=2{\lceil \frac{2l+1}{2}\rceil}+2d-2l=2d+2.$ \\
- \ If $|\alpha|=2d+1$, then $|\gamma|=2l+1$ with $l\in\naturals_0$. Thus, $\deg(p_\alpha)=2{\lceil \frac{2l+2}{2}\rceil}+2d+1-2l-1=2d+2$. 

Last but not least, if $\gamma\equiv 0$ then $p_\alpha(\ux):=\prod_{i=1}^n(1+X_i^2)^{\beta_i}$ satisfies all the required properties.
 \end{proof}
\begin{crl}\label{poly-init-sgmnt}
Let $\mathcal{P}\subset\rx$ be a family of polynomials of degree at most $k$. 
Then there exists a polynomial $p$ of degree $k+1$ when $k$ is odd and $k+2$ if $k$ is even, such that $p\geq 1$ on $\reals^n$ and $\sup_{\underline{y}\in\reals^n}\left|\frac{f(\underline{y})}{p(\underline{y})}\right|<\infty$ for 
all $f\in\mathcal{P}$.
\end{crl}
\proof Apply Lemma \ref{min-bnd-poly} to each monomial $\ux^\alpha$ appearing in $\mathcal{P}$ and set $p:=1+\sum_{\alpha}p_\alpha$.
\endproof
Recall that, for any $k\in \naturals$, we denote by $\reals[\ux]_k$ the space of all polynomials in $\reals[\ux]$ of degree at most $k$.

\begin{crl}\label{poly-one-degree-higher}
Let $K\subseteq \reals^n$ be closed, $d \in \naturals$, and $L$ be a $K$--positive linear functional on $\reals[\ux]_{2d}$ (resp. $\reals[\ux]_{2d+1}$). Then there exists $p\in\reals[\ux]_{2d+2}$ such that $p\geq 1$ on $K$ and $\sup_{\underline{y}\in\reals^n}\left|\frac{f(\underline{y})}{p(\underline{y})}\right|<\infty$ for 
all $f\in\reals[\ux]_{2d}$ (resp. for 
all $f\in\reals[\ux]_{2d+1}$). Hence, there exists a $K$--representing measure for $L$ if and only if $L$ admits a $K$--positive extension to $B_p$.
\end{crl}
\begin{proof}
Let $k\in\{2d, 2d+1\}$. By Corollary \ref{poly-init-sgmnt} applied to $\mathcal{P}=\reals[\ux]_{k}$, we know that there exists $p\in\reals[\ux]_{2d+2}$ such that $p\geq 1$ on $\mathbb{R}^n$ and 
$\sup_{\underline{y}\in K}\left|\frac{b(\underline{y})}{p(\underline{y})}\right|<\infty$ for all $b\in\reals[\ux]_k$. Then we can apply Theorem \ref{trnctmmnt} for $B=\mathcal{P}$ and
such a $p$ and get exactly the desired statement.\end{proof}

Corollary \ref{poly-one-degree-higher} slightly improves Curto and Fialkow's result \cite[Theorem 2.2]{Curto-Fialkow-2008}, because the first requires the existence of a $K-$positive extension to a subspace $B_p$ of $\reals[\ux]_{2d+2}$ while the second to the whole $\reals[\ux]_{2d+2}$.

When $K$ is compact, using Theorem \ref{cmpttrnctmmnt} we can also partly retrieve the following result due to Fialkow and Nie (see \cite[Theorem 2.2]{Fialkow-Nie}), which is a generalization of a well-known result by Tchakaloff  (see \cite[Theorem I, p.~129]{Tchakaloff}). Note that we cannot directly derive from Theorem \ref{cmpttrnctmmnt} that the representing measure is finitely atomic, but only that there exists a $K$--representing measure. However, the reader might find useful to see it as a consequence of our results.
 
\begin{crl} \label{FiNie100}
Suppose $K\subset\reals^{n}$ is compact, $d \in \naturals$, $H$ is a linear subspace of $\reals[\ux]_d$ such that there exists $p\in H$ with $p(x)>0$ for all $x\in K$. Let $L:H\to \reals$ be a linear functional that is $K$--positive. Then there exists a $K$--representing measure for $L$.
\end{crl}

\begin{proof}
By applying Theorem \ref{cmpttrnctmmnt} for $A=\reals[\ux]$, $B=H$ and $K\subset\reals^n$ compact, we get the desired conclusion.
\end{proof}

In \cite{Fialkow-Nie}, the linear subspace $H$ in Corollary \ref{FiNie100} is called $K$--full. 


\subsection{{Triangular, Rectangular, Sparse Connected Truncated Moment Problems}} \ We now study some natural connections with the TMP considered by R.E. Curto and L.A. Fialkow (\cite{tcmp1,tcmp3}), by M. Putinar \cite{Put}, by J. Nie \cite{Nie}, and by M. Laurent and B. Mourrain \cite{LaMo}. 
For convenience, we illustrate these special instances of TMP for the two-variable case, namely we present them as $B$--truncated $K$--moment problems with $B\subseteq \reals[s,t]$ and $K\subseteq\mathbb{R}^2$ closed (cf. Definition \ref{B-TKMP}). 

Given a collection of real numbers $\gamma \equiv \gamma ^{(2n)}=\{\gamma_{00},\gamma _{10},\gamma _{01},\cdots ,$ $\gamma _{2n,0},$\ $\gamma
_{2n-1,1},\cdots , \gamma _{1,2n-1}, \gamma _{0,2n}\}$, with $\gamma _{00}=1$ and $K\subseteq\reals^2$ closed, the Classical $K-$TMP consists of finding a positive Radon measure $\mu $
supported in $K$ such that 
$$
\gamma _{ij}=\int s^{i}t^{j}\;d\mu \;\;\;(0\leq i+j\leq 2n);
$$
in this case, $\gamma $ is called a \textit{$K-$truncated moment sequence} (of order $2n$) and $\mu $ is called a 
\textit{$K$-representing measure} for $\gamma $.
Hence, the Classical $K-$TMP corresponds to the $B$--truncated $K$--moment problem with $B=\spn\{s^{i}t^{j}: 0\leq i+j\leq 2n\}=\mathbb{R}[s,t]_{2n}$.

The reader will easily observe that the powers of monomials in $B=\mathbb{R}[s,t]_{2n}$ can be arranged in a right-triangle configuration in $\naturals_0 \times \naturals_0$ (see Figure \ref{Figure 1}, left diagram); this is why the Classical TMP is often addressed to as a \textit{Triangular TMP}. The \textit{Rectangular TMP} considered by M. Putinar \cite{Put} instead corresponds to the right-hand side diagram in Figure \ref{Figure 1} (see also \cite{KimWoe}); that is, one considers $B=\spn\{s^it^j: 0 \le i \le M, \ 0 \le j \le N\}$ for fixed $M,N\in\mathbb{N}$. 

\psset{unit=0.7mm, linewidth=0.7\pslinewidth} 
\begin{figure}[h]
\begin{center}
\begin{pspicture}(165,65)

\rput(0,12){

\psline(10,20)(75,20)
\psline(10,40)(75,40)
\psline(10,60)(75,60)
\psline(10,80)(75,80)
\psline(10,20)(10,85)
\psline(30,20)(30,85)
\psline(50,20)(50,85)
\psline(70,20)(70,85)

\pscircle[fillstyle=solid,fillcolor=black!10](10,20){4}
\put(9,19){$1$} 
\pscircle[fillstyle=solid,fillcolor=black!10](30,20){4}
\put(29,19){$s$}
\pscircle[fillstyle=solid,fillcolor=black!10](50,20){4}
\put(48,18.5){$s^2$}
\pscircle[fillstyle=solid,fillcolor=black!10](10,40){4}
\put(9,39){$t$}
\pscircle[fillstyle=solid,fillcolor=black!10](10,60){4}
\put(8,58.5){$t^2$}
\pscircle[fillstyle=solid,fillcolor=black!10](30,40){4}
\put(28,39){$st$}


}

\rput(80,2.5){

\psline(10,20)(75,20)
\psline(10,40)(75,40)
\psline(10,60)(75,60)
\psline(10,80)(75,80)
\psline(10,20)(10,85)
\psline(30,20)(30,85)
\psline(50,20)(50,85)
\psline(70,20)(70,85)

\pscircle[fillstyle=solid,fillcolor=black!10](10,20){4}
\put(9,19){$1$} 
\pscircle[fillstyle=solid,fillcolor=black!10](30,20){4}
\put(29,19){$s$}
\pscircle[fillstyle=solid,fillcolor=black!10](50,20){4}
\put(47.5,18.5){$s^2$}
\pscircle[fillstyle=solid,fillcolor=black!10](10,40){4}
\put(9,39){$t$}
\pscircle[fillstyle=solid,fillcolor=black!10](30,40){4}
\put(28,39){$st$}
\pscircle[fillstyle=solid,fillcolor=black!10](50,40){4}
\put(46.6,39){$s^2t$}


}

\end{pspicture}
\end{center}
\caption{\small{Diagrams of monomial powers in the Classical Truncated Moment Problem (left) and the Rectangular Truncated Moment Problem (right)}}
\label{Figure 1}
\end{figure}

Both diagrams are actually special cases of the \textit{Sparse Connected} TMP considered by M. Laurent and B. Mourrain \cite{LaMo}. For, they focus on monomial diagrams which are associated with sets $\mathcal{C}$ which are both \textit{connected} and contain the monomial $1$, that is a $B$--truncated moment problem for $B=\spn(\mathcal{C})$ in our terminology. A set $\mathcal{C}$ is connected if every monomial in $\mathcal{C}$ is the endpoint of a staircase path starting at $1$. For instance, $\{1,s,st\}$ is connected (see Figure \ref{Figure 2}), but $\{1,st\}$ is not. 

\psset{unit=0.7mm, linewidth=.75\pslinewidth} 
\begin{figure}[h]
\begin{center}
\begin{pspicture}(165,70)

\rput(38,5){

\psline(10,20)(75,20)
\psline(10,40)(75,40)
\psline(10,60)(75,60)
\psline(10,80)(75,80)
\psline(10,20)(10,85)
\psline(30,20)(30,85)
\psline(50,20)(50,85)
\psline(70,20)(70,85)

\pscircle[fillstyle=solid,fillcolor=black!10](10,20){4}
\put(9,19){$1$} 
\pscircle[fillstyle=solid,fillcolor=black!10](30,20){4}
\put(29,19){$s$}
\pscircle[fillstyle=solid,fillcolor=black!10](30,40){4}
\put(28,39){$st$}

}





\end{pspicture}
\end{center}
\caption{\small{Diagrams of monomials for $\mathcal{C}=\{1,s,st\}$}}
\label{Figure 2}
\end{figure}

Our general setting can also cover the Sparse Connected TMP in the more general case when there are infinitely many moments given in one of the variables. This is illustrated in Figure \ref{Figure 3}, where we consider the case $\mathcal{C}:=\{1,s,t,st,t^2,s^2t,s^3t,\cdots\}$.

For all the situations described so far, if we assume that the goal is to find a representing measure supported in a compact $K\subseteq\mathbb{R}^2$, then the hypotheses in Theorem \ref{cmpttrnctmmnt} are satisfied, as the monomial $1$ is always in $B$ and $K$ is compact. Therefore, the Riesz functional $\Lambda_{\gamma}$ admits an integral representation via a positive Radon measure supported in $K$, provided one can verify that $\Lambda_{\gamma}$ is nonnegative when evaluated on polynomials nonnegative on $K$. (Recall that the Riesz functional $\Lambda_{\gamma}:B \subseteq \mathbb{R}[s,t]_n \rightarrow \reals$ is given by $p \, {\longmapsto} \sum_{ij}p_{ij}\gamma _{ij}$ for all $p(s,t)\equiv \sum_{ij}p_{ij}s^{i}t^{j}$.) \ This verification varies from case to case, but it typically comes down to checking that, in addition to the positivity of the basic moment matrix associated with $\gamma$ one needs the positivity of the so-called localizing matrices, which keep track of the support of the representing measure (see \cite{tcmp3}, \cite{tcmp4}, \cite{Curto-Fialkow-2005}). 

In summary, when $K$ is compact, as an application of Theorem \ref{cmpttrnctmmnt} we can subsume some important aspects of previous results on the existence of $K-$representing measures for the Classical (i.e., {\it Triangular}) TMP, Putinar's {\it Rectangular} TMP, Nie's $\mathcal{A}$--TMP, and the {\it Sparse Connected} case covered in \cite{LaMo}. All of these TMPs are concerned with finite collections of moments, but our results do allow for {\it infinitely} many moments as initial data, see e.g. in Figure 4.

\psset{unit=.75mm, linewidth=.75\pslinewidth} 
\begin{center}
\begin{figure}[ht]
\begin{pspicture}(85,75)

\rput(0,12){

\psline(10,20)(86,20)
\psline(10,40)(86,40)
\psline(10,60)(86,60)
\psline(10,80)(86,80)
\psline(10,20)(10,93)
\psline(30,20)(30,93)
\psline(50,20)(50,93)
\psline(70,20)(70,93)

\pscircle[fillstyle=solid,fillcolor=black!10](10,20){4}
\put(9,19){$1$} 
\pscircle[fillstyle=solid,fillcolor=black!10](30,20){4}
\pscircle[fillstyle=solid,fillcolor=black!10](10,60){4}
\put(8.6,58.3){$t^2$}
\pscircle[fillstyle=solid,fillcolor=black!10](10,40){4}
\put(29,19){$s$}
\pscircle[fillstyle=solid,fillcolor=black!10](30,40){4}
\put(28,39){$st$}
\pscircle[fillstyle=solid,fillcolor=black!10](50,40){4}
\put(46.7,38.8){$s^2t$}
\pscircle[fillstyle=solid,fillcolor=black!10](70,40){4}
\put(66.7,38.8){$s^3t$}
\put(87,39){$\boldsymbol{\cdots}$}


\put(9,39){$t$}
}

\end{pspicture}
\caption{\small{Monomial diagram for $\mathcal{C}=\{1,s,t,st,t^2,s^2t,s^3t,\cdots\}$} }
\label{Figure 3}
\end{figure}
\end{center}

When $K$ is non compact, Corollary \ref{poly-one-degree-higher} provides for the Triangular TMP that the Riesz functional $\Lambda_{\gamma}: \mathbb{R}[s,t]_{n}\to\mathbb{R}$ admits an integral representation via a positive Radon measure supported in $K$ if and only if   $\Lambda_{\gamma}$ has a $K-$positive extension to a certain proper subspace of $\mathbb{R}[s,t]_{n+2}$. A similar result can be easily showed for the Rectangular and Sparse Connected $K-$TMP for $K$ non-compact and $\mathcal{C}$ finite, combining Corollary \ref{poly-init-sgmnt} and Theorem \ref{trnctmmnt}. However, this technique will not always provide a solution to the Sparse Connected $K-$TMP for $K$ non-compact and $\mathcal{C}$ infinite, except when the unboundedness of the degree in one variable is compensated by a compact component of $K$. For example, if we take $\mathcal{C}=\{1,s,t,st,t^2,s^2t,s^3t,\cdots\}\subseteq\mathbb{R}[s][t]_{2}$ and $K=K_1\times \mathbb{R}$ with $K_1\subset\mathbb{R}$ compact then $L: \spn(\mathcal{C})\to \mathbb{R}$ has a $K-$representing measure if and only if $L$ has a $K-$positive extension to $B_p$ for any $p\in\mathbb{R}[s][t]_{3}$ nonnegative on $K$ (see Theorem \ref{thm-partially-cmpct}). 

We momentarily digress to recall some basic facts about moment matrices and flat extensions. Naturally associated with each TMP is a moment matrix $M \equiv M(n)\equiv M(n)(\gamma )$, given by $M_{(i_1,i_2),(j_1,j_2)}:=\gamma_{(i_1+j_1,i_2+j_2)}$. For instance, when $n=3$,  
\begin{equation*}
M(3)\equiv \left( 
\begin{array}{cccccc}
\gamma_{00} & \gamma_{10} & \gamma_{01} & \gamma_{20} & \gamma_{11} & \gamma_{02} \\ 
\gamma_{10} & \gamma_{20} & \gamma_{11} & \gamma_{30} & \gamma_{21} & \gamma_{12} \\ 
\gamma_{01} & \gamma_{11} & \gamma_{02} & \gamma_{21} & \gamma_{12} & \gamma_{03} \\
\gamma_{20} & \gamma_{30} & \gamma_{21} & \gamma_{40} & \gamma_{31} & \gamma_{22} \\ 
\gamma_{11} & \gamma_{21} & \gamma_{12} & \gamma_{31} & \gamma_{22} & \gamma_{13} \\ 
\gamma_{02} & \gamma_{12} & \gamma_{03} & \gamma_{22} & \gamma_{13} & \gamma_{04} 
\end{array}%
\right). 
\end{equation*}
Observe that $M(n+1)=\left(\begin{array}{cccc} M(n) & B \\ B^{*} & C \end{array} \right)$, and recall
that $M(n+1) \geq 0 \Leftrightarrow \textrm{(i)} M(n) \geq 0$, (ii) $B=M(n)W$ for some $W$, and (iii) $C \geq W^{*}M(n)W$ 
(\cite{Smu}; cf. \cite{tcmp1}). 

\medskip
The matrix $\ M(n)$ includes all given moments at least once. The positive semidefinetess of $M(n)$ (as a Hilbert space operator on a finite dimensional space) corresponds to the positivity of the Riesz functional on the positive cone generated by polynomials of the form $p^2$, where $p \in \mathbb{R}[s,t]_n$. Let $r:=\rank M(n)$. A fundamental result in TMP theory is the so-called Flat Extension Theorem, which states that $\gamma^{(2n)}$ admits an $r$--atomic representing measure if and only if $M(n)$ admits a \textit{flat} extension $M(n+1)$; that is an extension such that $\rank M(n+1)=\rank M(n)=r$. In \cite{tcmp1} the columns of $M(n)$ are labeled using monomials as indices, as follows: $1,S,T,S^2,ST,T^2,S^3,S^2T,ST^2,T^3,\cdots$. In this terminology, to say that $M(n+1)$ is a flat extension of $M(n)$ means that each column labeled with a monomial of degree $n+1$ can be written as a linear combination of columns labeled with monomials of degree at most $n$. 

\smallskip
In both diagrams in Figure \ref{Figure 5} the monomials in the squares represent the columns needed to extend $M(n)$ and generate a flat extension in the Classical and Rectangular TMP, respectively. In the case of the Sparse Connected TMP, the transition from a matrix containing the original moments to a flat extension is provided by the so-called border of $\mathcal{C}$. This is the set $\mathcal{C}^+ \setminus \mathcal{C}$, where $\mathcal{C}^+:=\mathcal{C} \cup s\mathcal{C}\cup t\mathcal{C}$. A visual description of these sets is given in Figure~\ref{Figure 6} (resp. Figure \ref{Figure 7}) for the set $\mathcal{C}$ of monomials represented in Figure \ref{Figure 2} (resp. Figure \ref{Figure 3}). 

\psset{unit=0.7mm, linewidth=0.7\pslinewidth} 
\begin{figure}[!ht]
\begin{center}
\begin{pspicture}(165,65)

\rput(0,22){

\psline(10,20)(75,20)
\psline(10,40)(75,40)
\psline(10,60)(75,60)
\psline(10,80)(75,80)
\psline(10,20)(10,85)
\psline(30,20)(30,85)
\psline(50,20)(50,85)
\psline(70,20)(70,85)

\pscircle[fillstyle=solid,fillcolor=black!10](10,20){4}
\put(9,19){$1$} 
\pscircle[fillstyle=solid,fillcolor=black!10](30,20){4}
\put(29,19){$s$}
\pscircle[fillstyle=solid,fillcolor=black!10](50,20){4}
\put(48,18.5){$s^2$}
\pscircle[fillstyle=solid,fillcolor=black!10](10,40){4}
\put(9,39){$t$}
\pscircle[fillstyle=solid,fillcolor=black!10](10,60){4}
\put(8.5,58.5){$t^2$}
\pscircle[fillstyle=solid,fillcolor=black!10](30,40){4}
\put(28,39){$st$}

\psframe[fillstyle=solid,fillcolor=black!28](67,17)(73,23)
\put(68,18.5){$s^3$} 
\psframe[fillstyle=solid,fillcolor=black!28](47,37)(53.6,43.7)
\put(47,39){$s^2t$}
\psframe[fillstyle=solid,fillcolor=black!28](27,57)(33.8,63.4)
\put(27.5,58.6){$st^2$}
\psframe[fillstyle=solid,fillcolor=black!28](7,77)(13,83)
\put(8.5,78.3){$t^3$}

}

\rput(80,12.5){

\psline(10,20)(75,20)
\psline(10,40)(75,40)
\psline(10,60)(75,60)
\psline(10,80)(75,80)
\psline(10,20)(10,85)
\psline(30,20)(30,85)
\psline(50,20)(50,85)
\psline(70,20)(70,85)

\pscircle[fillstyle=solid,fillcolor=black!10](10,20){4}
\put(9,19){$1$} 
\pscircle[fillstyle=solid,fillcolor=black!10](30,20){4}
\put(29,19){$s$}
\pscircle[fillstyle=solid,fillcolor=black!10](50,20){4}
\put(47.6,18.6){$s^2$}
\pscircle[fillstyle=solid,fillcolor=black!10](10,40){4}
\put(9,39){$t$}
\pscircle[fillstyle=solid,fillcolor=black!10](30,40){4}
\put(28,39){$st$}
\pscircle[fillstyle=solid,fillcolor=black!10](50,40){4}
\put(46,38.5){$s^2t$}

\psframe[fillstyle=solid,fillcolor=black!28](67,17)(73,23)
\put(68,18.5){$s^3$} 
\psframe[fillstyle=solid,fillcolor=black!28](67,37)(74,43.5)
\put(67,39){$s^3t$}
\psframe[fillstyle=solid,fillcolor=black!28](47,57.5)(55.5,64)
\put(47,58.5){$s^2t^2$}
\psframe[fillstyle=solid,fillcolor=black!28](27,57)(33.4,63.4)
\put(27,58.5){$st^2$}
\psframe[fillstyle=solid,fillcolor=black!28](7,57)(13,63)
\put(8.5,58.5){$t^2$}

}

\end{pspicture}
\end{center}
\caption{\small{Diagrams of monomial powers in the Classical Truncated Moment Problem (left, circles) and the Rectangular Truncated Moment Problem (right, circles), and for their respective flat extensions (circles and squares)}}
\label{Figure 5}
\end{figure}



\psset{unit=0.7mm, linewidth=.75\pslinewidth} 
\begin{figure}[!ht]
\begin{center}
\begin{pspicture}(165,70)

\rput(-2,20){

\psline(10,20)(75,20)
\psline(10,40)(75,40)
\psline(10,60)(75,60)
\psline(10,80)(75,80)
\psline(10,20)(10,85)
\psline(30,20)(30,85)
\psline(50,20)(50,85)
\psline(70,20)(70,85)

\pscircle[fillstyle=solid,fillcolor=black!10](10,20){4}
\put(9,19){$1$} 
\pscircle[fillstyle=solid,fillcolor=black!10](30,20){4}
\put(29,19){$s$}
\pscircle[fillstyle=solid,fillcolor=black!10](30,40){4}
\put(28,39){$st$}

\psframe[fillstyle=solid,fillcolor=black!28](47,17)(53.4,23.6)
\put(48.5,19){$s^2$} 
\psframe[fillstyle=solid,fillcolor=black!28](47,37)(53.7,43)
\put(47,39){$s^2t$}
\psframe[fillstyle=solid,fillcolor=black!28](7,37)(13,43)
\put(9,39){$t$}
\psframe[fillstyle=solid,fillcolor=black!28](27,57)(33.5,63.5)
\put(27,59){$st^2$}

}

\rput(77,20){

\psline(10,20)(75,20)
\psline(10,40)(75,40)
\psline(10,60)(75,60)
\psline(10,80)(75,80)
\psline(10,20)(10,85)
\psline(30,20)(30,85)
\psline(50,20)(50,85)
\psline(70,20)(70,85)

\pscircle[fillstyle=solid,fillcolor=black!10](10,20){4}
\put(9,19){$1$} 
\pscircle[fillstyle=solid,fillcolor=black!10](30,20){4}
\put(29,19){$s$}
\pscircle[fillstyle=solid,fillcolor=black!10](50,20){4}
\put(48,18.5){$s^2$}
\pscircle[fillstyle=solid,fillcolor=black!10](10,40){4}
\put(9,39){$t$} 
\pscircle[fillstyle=solid,fillcolor=black!10](30,40){4}
\put(28,39){$st$}
\pscircle[fillstyle=solid,fillcolor=black!10](50,40){4}
\put(47,38.5){$s^2t$}
\pscircle[fillstyle=solid,fillcolor=black!10](30,60){4}
\put(27,58.5){$st^2$}


}

\end{pspicture}
\end{center}
\caption{\small{Diagrams of monomials for {$\mathcal{C}=\{1,s,st\}$} (left, circles), $\mathcal{C}^+=\{1,s,st,t,s^2,s^2t,st^2\}$ (right) and $\mathcal{C}^+\setminus\mathcal{C}$ (left, squares)}}
\label{Figure 6}
\end{figure}

\psset{unit=.75mm, linewidth=.75\pslinewidth} 
\begin{figure}[!ht]
\begin{center}
\begin{pspicture}(85,75)

\rput(0,22){

\psline(10,20)(86,20)
\psline(10,40)(86,40)
\psline(10,60)(86,60)
\psline(10,80)(86,80)
\psline(10,20)(10,93)
\psline(30,20)(30,93)
\psline(50,20)(50,93)
\psline(70,20)(70,93)

\pscircle[fillstyle=solid,fillcolor=black!10](10,20){4}
\put(9,19){$1$} 
\pscircle[fillstyle=solid,fillcolor=black!10](30,20){4}
\pscircle[fillstyle=solid,fillcolor=black!10](10,60){4}
\put(8.5,58.5){$t^2$}
\pscircle[fillstyle=solid,fillcolor=black!10](10,40){4}
\put(29,19){$s$}
\pscircle[fillstyle=solid,fillcolor=black!10](30,40){4}
\put(28,39){$st$}
\pscircle[fillstyle=solid,fillcolor=black!10](50,40){4}
\put(47,39){$s^2t$}
\pscircle[fillstyle=solid,fillcolor=black!10](70,40){4}
\put(67,39){$s^3t$}
\put(87,39){$\boldsymbol{\cdots}$}

\psframe[fillstyle=solid,fillcolor=black!28](7,77)(13,83)
\put(8.5,78.6){$t^3$} 
\psframe[fillstyle=solid,fillcolor=black!28](26,57)(33,63)
\put(27,58.6){$st^2$}
\psframe[fillstyle=solid,fillcolor=black!28](46,57)(54.6,63)
\put(46.3,59){$s^2t^2$}
\psframe[fillstyle=solid,fillcolor=black!28](66,57)(75,63)
\put(67,58.6){$s^3t^2$}
\psframe[fillstyle=solid,fillcolor=black!28](47,17)(53,23)
\put(48.2,18.5){$s^2$}
\put(87,59){$\boldsymbol{\cdots}$}

\put(9,39){$t$}
}

\end{pspicture}
\end{center}
\caption{\small{Monomial diagram for {$\mathcal{C}=\{1,s,t,st,t^2,s^2t,s^3t,\ldots\}$} (circles) and {$\mathcal{C}^+ \setminus \mathcal{C}=\{s^2,st^2,t^3,s^2t^2,s^3t^2,\ldots\}$ (squares)}.} }
\label{Figure 7}
\end{figure}

The reader will notice that, while our Theorem \ref{trnctmmnt} promotes the given linear functional $L$ from the subspace $B$ to $B_p$ (thus increasing the linear dimension of the subspace by $1$ or, alternatively, by requiring that $\dim \ B_p \setminus B =1$), the Flat Extension Theorem demands both a finite dimensional $B$ together with a collection of polynomials that describe each monomial of maximum degree in terms of monomials of lower degree. On the one hand, to use Theorem \ref{trnctmmnt} we need to find only one polynomial $p$ satisfying the key property (\ref{p-existence}). On the other hand, we need to find several polynomials to claim flatness; however, the moment matrix $M(n)$ holds the key, since each such polynomial arises from a linear equation involving the columns of the moment matrix. How the two types of results mesh together is an intriguing matter. In future work, we plan to return to this matter, and focus on the connections between our results here and the Flat Extension Theorem. 


\subsection {{Connections with the Subnormal Completion Problem}} \ We end this section with a natural connection to the SCP for $2$--variable weighted shifts, studied by R.E. Curto, S.H. Lee and J. Yoon \cite{CLY} and by D. Kimsey (\cite{Kimsey2014,Kimsey2016}). To easily exemplify the connections, we will focus on TMP whose representing measures (when they exist) have support in the nonnegative quadrant $\mathbb{R}_+ \times \mathbb{R}_+$. For the relevant terminology and basic results, we refer the reader to \cite{CLY} and \cite{Kimsey2016}. Very briefly, given an initial (finite) family of weights one first forms the associated moments. The SCP then asks for the existence of a positive Radon probability measure on the unit square which interpolates these moments. Now recall that, given two double-indexed positive sequences of weights $\alpha _{\mathbf{k}},\beta _{\mathbf{k}}\in \ell ^{\infty }(\mathbb{Z}_{+}^{2})$, both bounded by $1$, the $2$--variable weighted shift $\mathbf{T}\equiv
(T_{1},T_{2})$\ acts on $\ell^2(\naturals_0^2)$ and is defined by 
\begin{equation*}
T_{1}e_{\mathbf{k}}:=\alpha _{\mathbf{k}}e_{\mathbf{k+}\varepsilon _{1}}, \ \ T_{2}e_{\mathbf{k}}:=\beta _{\mathbf{k}}e_{\mathbf{k+}\varepsilon _{2}},
\end{equation*}
where $\{e_{\mathbf{k}}\}_{k \in \naturals_0^2}$ is the canonical orthonormal basis for $\ell^2(\naturals_0^2)$, and $\mathbf{\varepsilon }_{1}:=(1,0)$ and $\mathbf{\varepsilon }%
_{2}:=(0,1)$. The pair $(T_1,T_2)$ is said to be subnormal if it is the restriction to $\ell^2(\naturals_0^2)$ of a commuting pair of normal operators. Thus, subnormality of $(T_1,T_2)$ requires that $T_1$ and $T_2$ commute, and it is easy to check that this is equivalent to 
\begin{equation} \label{commuting}
\beta _{\mathbf{k+}\varepsilon_{1}}\alpha _{\mathbf{k}}=\alpha _{\mathbf{k+}\varepsilon _{2}}\beta _{\mathbf{k}}\;\;\,\text{for all }\mathbf{k} \in \naturals_0^2.
\end{equation}
Subnormality of $(T_1,T_2)$ also implies that each $T_i$ is subnormal ($i=1,2$). Hence, for each $j \in \naturals_0$ the sequence $\{\alpha_{(m,j)}\}_{m=0}^{\infty}$ must be associated with a subnormal unilateral weighted shift. Similarly, for each $i \in \naturals_0$ the sequence $\{\beta_{(i,n)}\}_{n=0}^{\infty}$ must be associated with a subnormal unilateral weighted shift. As a result, the subnormality of the pair $(T_1,T_2)$ gives rise to the subnormality of two infinite families of unilateral weighted shifts, one family corresponding to the rows of the associated weight diagram for $(T_1,T_2)$, and the other family corresponding to the columns. 

A natural question arises: beyond the commutativity of $T_1$ and $T_2$, and the subnormality of each row and each column in the weight diagram for $(T_1,T_2)$, what else is needed for the pair to be subnormal? \ This is the so-called Lifting Problem for Commuting Subnormals, which has been studied extensively by R.E. Curto, J. Yoon and others; see, for instance, the recent survey \cite{CurtoSurvey}. 

We recall here a well known characterization of subnormality for unilateral weighted shifts, due to C. Berger (cf. \cite[III.8.16]{Con}) and independently established by R. Gellar and L.J. Wallen \cite{GeWa}: \ $W_{\omega }$ is subnormal if and only if there exists a probability measure $\xi $ supported in $[0,\left\| W_{\omega }\right\| ^{2}]$ such that 
$$
\gamma_{k}(\omega ):=\omega _{0}^{2}\cdot ...\cdot \omega _{k-1}^{2}=\int
t^{k}\;d\xi (t)\;\;\, \textrm{for all } k \in \naturals.
$$
The sequence $\gamma_k(\omega)$ is called the sequence of moments of $\omega$. For $2$--variable weighted shifts, the appropriate generalization is due to N.P. Jewell and A.R. Lubin \cite{JeLu}: \ A $2$--variable weighted shift $(T_{1},T_{2})$ admits a commuting normal extension if and only if there is a probability measure $\mu $
defined on the $2$--dimensional rectangle $R=[0,a_{1}]\times \lbrack 0,a_{2}]$
($a_{i}:=\left\| T_{i}\right\| ^{2}$) such that $\gamma _{\mathbf{k}%
}=\int_{R}\mathbf{t}^{\mathbf{k}}d\mu (\mathbf{t}):=\int_{R}t_{1}^{k_{1}}t_{2}^{k_{2}}\;d\mu (t_{1},t_{2})$ \ $(\textrm{for all } \mathbf{%
k\in }\naturals_0^{2}$), where the moments $\gamma_{\mathbf{k}}$ are defined by
\begin{equation*} 
\gamma _{\mathbf{k}}(\alpha ,\beta )\!:=\!\left\{ 
\!\begin{array}{cc}
1 & \text{if }\mathbf{k}=0 \\ 
\alpha _{(0,0)}^{2}\cdot ...\cdot \alpha _{(k_{1}-1,0)}^{2}\cdot \beta
_{(k_{1},0)}^{2}\cdot ...\cdot \beta _{(k_{1},k_{2}-1)}^{2} & \text{if }
\mathbf{k\in }\mathbb{Z}_{+}^{2}\text{, }\mathbf{k}\neq 0
\end{array}
\right\} \!.
\end{equation*} 

\vspace{0.2cm}

\begin{dfn}
\label{subnormalcompletion}Given $m\geq 0$ and a finite family of positive
numbers $\Omega _{m}\equiv \{(\alpha _{\mathbf{k}},\beta _{\mathbf{k}%
})\}_{\left| \mathbf{k}\right| \leq m}$, we say that a $2$--variable weighted
shift $\mathbf{T}\equiv (T_{1},T_{2})$ with weight sequences $\alpha _{%
\mathbf{k}}^{\mathbf{T}}$ and $\beta _{\mathbf{k}}^{\mathbf{T}}$ is a
subnormal completion of $\Omega _{m}$ if (i) $\mathbf{T}$ is subnormal, and
(ii) $(\alpha _{\mathbf{k}}^{\mathbf{T}},\beta _{\mathbf{k}}^{\mathbf{T}%
})=(\alpha _{\mathbf{k}},\beta _{\mathbf{k}})$ whenever $\left| \mathbf{k}%
\right| \leq m$.
\end{dfn}

\begin{exm}
When $m=1$, we shall let $a:=\alpha _{00}^{2}$, $b:=\beta _{00}^{2}$, $%
c:=\alpha _{10}^{2}$, $d:=\beta _{01}^{2}$, $e:=\alpha _{01}^{2}$ and $%
f:=\beta _{10}^{2}$. To be consistent with the commutativity of a $2$%
--variable weighted shifts whose weight sequences satisfy (\ref{commuting}),
we shall always assume $af=be$. The moments of $\Omega _{1}$ are 
\begin{equation*}
\left\{ 
\begin{array}{ccc}
\gamma _{00}:=1 &  &  \\ 
\gamma _{01}:=a & \gamma _{10}:=b &  \\ 
\gamma _{02}:=ac & \gamma _{11}:=be & \gamma _{20}:=bd%
\end{array}%
\right. ,
\end{equation*}%
and the associated moment matrix is 
\begin{equation} \label{M1Omega}
M(\Omega _{1}):=\left( 
\begin{array}{ccc}
1 & a & b \\ 
a & ac & be \\ 
b & be & bd%
\end{array}%
\right) .
\end{equation}%
In this case, solving the SCP consists of finding a probability measure $\mu 
$ supported in $[0,\left\|T_1\right\|^2] \times [0,\left\|T_2\right\|^2] \subseteq [0,1]^2$  such that $\int_{\mathbb{R}%
_{+}^{2}}s^{i}t^{j}\;d\mu (s,t)=\gamma _{ij}\;(i,j\geq 0,\;i+j\leq 2)$. \qed
\end{exm}

In what follows, and for simplicity, we will specialize to the case $m=1$ in two variables, and show that the
condition $M(\Omega _{1})\geq 0$ is sufficient for the existence of a subnormal completion.

\begin{thm}
\label{quartic} (\cite[Theorem 5.1]{CLY}) \ Let $\Omega _{1}$ be a quadratic, commutative, initial set of
positive weights, and assume $M(\Omega _{1})\geq 0$. Then there exists a quartic commutative extension $\hat{\Omega}_{3}$ of $\Omega _{1}$
such that $M(\hat{\Omega}_{3})$ is a flat extension of $M(\Omega _{1})$, and 
$M_{s}(\hat{\Omega}_{3})\geq 0$ and $M_{t}(\hat{\Omega}_{3})\geq 0$. As a
consequence, $\Omega _{1}$ admits a subnormal completion $\mathbf{T}_{\hat{%
\Omega}_{\infty }}$. (The family $\Omega _{1}$ is shown in Figure~\ref{initial}.)
\end{thm}

\setlength{\unitlength}{0.8mm} \psset{unit=0.8mm} 
\begin{figure}[h]
\begin{center}
\begin{picture}(50,40)

\psline(0,0)(40,0)
\psline(0,20)(20,20)
\psline(0,0)(0,40)
\psline(20,0)(20,20)

\put(8,2){\footnotesize{$\sqrt{a}$}}
\put(28,2){\footnotesize{$\sqrt{c}$}}

\put(8,22){\footnotesize{$\sqrt{e}$}}

\put(1,9){\footnotesize{$\sqrt{b}$}}
\put(1,29){\footnotesize{$\sqrt{d}$}}
\put(21,9){\footnotesize{$\sqrt{f}$}}

\end{picture}
\end{center}
\caption{\small{The initial family of weights $\Omega _{1}$}}
\label{initial}
\end{figure}

The proof of Theorem \ref{quartic} given in \cite{CLY} uses localizing matrices to identify entries of a proposed moment matrix extension of the initial moment matrix $M(\Omega_1)$. However, the SCP is a special version of the $[0,1]^2$--TMP, with the extra requirement that all moments must be positive, and the support of a representing measure must be in the unit square. Therefore, Theorem \ref{cmpttrnctmmnt} will again imply the existence of a representing measure once the positivity of the associated Riesz functional on nonnegative polynomials on the unit square can be verified. 

Moreover, with the aid of Theorem \ref{cmpttrnctmmnt}, Theorem \ref{quartic} admits a substantial generalization to the case of infinitely many weights located in any finite number of rows and columns, provided that the natural assumptions on the subnormality of the unilateral weighted shifts associated with those rows and columns are made. For instance, consider the case shown in Figure \ref{Figure 8}, in which two rows (the first and the third) are infinite and generate subnormal unilateral weighted shifts $W^{(1)}$ and $W^{(3)}$. 

\setlength{\unitlength}{0.85mm} \psset{unit=0.85mm} 
\begin{figure}[h]
\begin{center}
\begin{picture}(105,88)

\rput(0,30){

\psline(10,20)(50,20)
\psline(10,40)(106,40)
\psline(10,60)(30,60)
\psline(10,80)(106,80)
\psline(10,20)(10,100)
\psline(30,20)(30,80)
\psline(50,20)(50,40)

\put(16,22){$\alpha_{(0,0)}$} 
\put(36,22){$\alpha_{(1,0)}$}
\put(16,42){$\alpha_{(0,1)}$}
\put(36,42){$\alpha_{(1,1)}$}
\put(56,42){$\alpha_{(2,1)}$}
\put(76,42){$\alpha_{(3,1)}$}
\put(96,42){$\boldsymbol{\cdots}$}

\put(16,62){$\alpha_{(0,2)}$}

\put(16,82){$\alpha_{(0,3)}$}
\put(36,82){$\alpha_{(1,3)}$}
\put(56,82){$\alpha_{(2,3)}$}
\put(76,82){$\alpha_{(3,3)}$}
\put(96,82){$\boldsymbol{\cdots}$}

\put(11,29){$\beta_{(0,0)}$} 
\put(31,29){$\beta_{(1,0)}$} 
\put(51,29){$\beta_{(2,0)}$} 
\put(11,49){$\beta_{(0,1)}$} 
\put(31,49){$\beta_{(1,1)}$} 
\put(11,69){$\beta_{(0,2)}$} 
\put(31,69){$\beta_{(1,2)}$} 
\put(11,49){$\beta_{(0,1)}$} 
\put(11,89){$\beta_{(0,3)}$} 
}
\end{picture}
\end{center}
\caption{\small{Weight diagram for the SCP with two infinite rows, the first and the third. (For each complete square, \ref{commuting} must hold.) In this case, the associated unilateral weighted shifts $W^{(1)}$ and $W^{(3)}$ must be assumed subnormal, as a necessary condition for the solubility of the SCP. } }
\label{Figure 8}
\vspace{-0.3cm}
\end{figure}

\medskip
\noindent Since $B=\{1,s,t,s^2,st,t^2\}$ and the polynomial $1$ is in $B$ (observe also that we always have the moment $\gamma_{(0,0)}=1$) and $K$ is compact (being contained in the rectangle $\left[0,\sup_{\mathbf{k}} \alpha_{\mathbf{k}} \right] \times \left[0,\sup_{\mathbf{k}} \beta_{\mathbf{k}}\right]$), the existence of a solution to the SCP depends on the positivity of the appropriate localized matrices (to keep track of $K$) or, alternatively, on the associated linear functional being positive on the relevant initial collection of nonnegative polynomials.

\bigskip

{\bf Acknowledgments}. The authors wish to thank the referee for a careful reading of the paper, and for many insightful comments and suggestions that helped improve this article. R. Curto was partially supported by a U.S. NSF Grant DMS-1302666 and by a 2017 Simons Foundation Visiting Professor award; he also received travel support from Baden-W\"urttemberg Stiftung. M. Infusino is member of the GNAMPA group of INdAM; she is also indebted to the Baden-Württemberg Stiftung for financial support of this work within the Eliteprogramme for Postdocs. The work of S. Kuhlmann was partially supported by the Ausschuss f\"ur Forschungsfragen (AFF) of the University of Konstanz.

\end{document}